\documentclass[12pt]{amsart}
\usepackage{float}
\usepackage{subfigure}
\usepackage{tikz}
\usepackage{tkz-euclide}
\usetikzlibrary{shapes.geometric}
\tikzset{
	rline/.style ={color = red, line width =1pt}
}
\tikzset{
	bline/.style ={color = blue, line width =1pt}
}
\tikzset{
	gline/.style ={color = green, line width =1pt}
}
\usepackage[colorlinks,
            linkcolor=black,
            anchorcolor=blue,
            citecolor=black]{hyperref}
\usepackage{cleveref}
\usepackage{graphicx}
\usepackage{amsmath}
\usepackage{amssymb}
\usepackage{setspace}
\usepackage{amsthm}
\usepackage{appendix}
\allowdisplaybreaks[4]
\usepackage{geometry}
\newtheorem{thm}{Theorem}[section]
\newtheorem{cor}[thm]{Corollary}

\newtheorem{lem}[thm]{Lemma}
\newtheorem{prop}[thm]{Proposition}
\theoremstyle{definition}
\newtheorem{defn}[thm]{Definition}
\theoremstyle{remark}
\newtheorem{rem}[thm]{Remark}
\theoremstyle{conclusion}

\theoremstyle{conjecture}

\numberwithin{equation}{section}

\newcommand{\be}{\begin{equation}}
\newcommand{\ee}{\end{equation}}

\begin{document}
\title[Finsler $N$-Laplacian Liouville equation in cones]{Anisotropic Finsler $N$-Laplacian Liouville equation in convex cones}

\author{Wei Dai, Changfeng Gui, YunPeng Luo}

\address{School of Mathematical Sciences, Beihang University (BUAA), Beijing 100191, P. R. China, and Key Laboratory of Mathematics, Informatics and Behavioral Semantics, Ministry of Education, Beijing 100191, P. R. China}
\email{weidai@buaa.edu.cn}

\address{Department of Mathematics, University of Macau, Macau SAR, P. R. China}
\email{changfenggui@um.edu.mo}

\address{School of Mathematical Sciences, Beihang University (BUAA), Beijing 100191, P. R. China, and School of Mathematics and Statistics, Central China Normal University, Wuhan, 430079, P. R. China}
\email{yunpengluo@buaa.edu.cn}

\thanks{Research of the first and third author are supported by the NNSF of China (No. 12222102), the National Science and Technology Major Project (2022ZD0116401) and the Fundamental Research Funds for the Central Universities. Research of the second author is supported by University of Macau research grants  CPG2024-00016-FST, CPG2025-00032-FST, SRG2023-00011-FST, MYRG-GRG2023-00139-FST-UMDF, UMDF Professorial Fellowship of Mathematics, Macao SAR FDCT 0003/2023/RIA1 and Macao SAR FDCT 0024/2023/RIB1.}

\begin{abstract}
We consider the anisotropic Finsler $N$-Laplacian Liouville equation
\[-\Delta ^{H}_{N}u=e^u \qquad {\rm{in}}\,\, \mathcal{C},\]
where $N\geq2$, $\mathcal{C}\subseteq\mathbb{R}^{N}$ is an open convex cone including $\mathbb{R}^{N}$, the half space $\mathbb{R}^{N}_{+}$ and $\frac{1}{2^{m}}$-space $\mathbb{R}^{N}_{2^{-m}}:=\{x\in\mathbb{R}^{N}\mid x_{1},\cdots,x_{m}>0\}$ ($m=1,\cdots,N$), and the anisotropic Finsler $N$-Laplacian $\Delta ^{H}_{N}$ is induced by a positively homogeneous function $H(x)$ of degree $1$. All solutions to the Finsler $N$-Laplacian Liouville equation with finite mass are completely classified. In particular, if $H(\xi)=|\xi|$, then the Finsler $N$-Laplacian $\Delta ^{H}_{N}$ reduces to the regular $N$-Laplacian $\Delta_N$. Our result is a counterpart in the limiting case $p=N$ of the classification results in \cite{CFR} for the critical anisotropic $p$-Laplacian equations with $1<p<N$ in convex cones, and also extends the classification results in \cite{CK,CL,CW,CL2,E} for Liouville equation in the whole space $\mathbb{R}^{N}$ to general convex cones. In our proof, besides exploiting the anisotropic isoperimetric inequality inside convex cones, we have also proved and applied the radial Poincar\'{e} type inequality (Lemma \ref{A1}), which are key ingredients in the proof and of their own importance and interests.
\end{abstract}
\maketitle {\small {\bf Keywords:} Classification of solutions; convex cones; anisotropic Finsler $N$-Laplacian; quasi-linear Liouville equation; anisotropic isoperimetric inequality inside convex cones; radial Poincar\'{e} type inequality. \\

}

\section{Introduction}

\subsection{Background and setting of the problem}
In this paper, we are mainly concerned with the following anisotropic Finsler $N$-Laplacian Liouville equation in convex cones with finite mass:
    \begin{equation}\label{eq:1.1}
      \left\{
          \begin{aligned}
          &-\Delta ^{H}_{N}u=e^u \quad\,\,\, &{\rm{in}} \,\, \mathcal{C}, \\
          &a(\nabla u)\cdot \nu =0 \quad\,\,\, &{\rm{on}} \,\, \partial\mathcal{C},\\
          &\int_{\mathcal{C}}e^u\mathrm{d}x<+\infty,
          \end{aligned}
          \right.
    \end{equation}
where $N\geq2$, $\mathcal{C}\subseteq\mathbb{R}^{N}$ is an open convex cone including $\mathbb{R}^{N}$, the half space $\mathbb{R}^{N}_{+}$ and $\frac{1}{2^{m}}$-space $\mathbb{R}^{N}_{2^{-m}}:=\{x\in\mathbb{R}^{N}\mid x_{1},\cdots,x_{m}>0\}$ ($m=1,\cdots,N$). From now on, without loss of generality, up to rotations and translations, we can write $\mathcal{C}=\mathbb{R}^{k}\times\mathcal{\widetilde{C}}$ with $k\in\{0,\cdots ,N\}$ and $\mathcal{\widetilde{C}}\subset\mathbb{R}^{N-k}$ is an open convex cone with vertex at the origin $0_{\mathbb{R}^{N-k}}$ which contains no lines. In particular, $\mathcal{C}=\mathbb{R}^{N}$ if $k=N$, $\mathcal{C}=\mathbb{R}^{N}_{+}$ if $k=N-1$ and $\mathcal{\widetilde{C}}=\mathbb{R}_{+}$, and $\mathcal{C}=\mathbb{R}^{N}_{2^{-m}}$ if $k=N-m$ and $\mathcal{\widetilde{C}}=\mathbb{R}^{m}_{2^{-m}}:=\{x\in\mathbb{R}^{m}\mid x_{1},\cdots,x_{m}>0\}$. The function $H:\mathbb{R}^N\rightarrow\mathbb{R}$ with $H\in C^2_{loc}(\mathbb{R}^N\setminus \{0\})$ is a gauge, i.e., a nonnegative, positively homogenous of degree $1$, convex function satisfying $H(\xi)>0$ on $\mathbb{S}^{N-1}:=\{\xi\in\mathbb{R}^{N}\mid\,|\xi|=1\}$. One can easily verify that
\begin{equation}\label{norm}
  c_{H}|\xi|\leq H(\xi)\leq C_{H}|\xi|, \qquad \forall \,\, \xi\in\mathbb{R}^{N},
\end{equation}
where $C_{H}:=\max\limits_{\xi\in \mathbb{S}^{N-1}}H(\xi)\geq c_{H}:=\min\limits_{\xi\in \mathbb{S}^{N-1}}H(\xi)>0$.

\medskip

Comparing with the regular Euclidean space $(\mathbb{R}^{N},|\cdot|)$, the anisotropic Finsler manifold $(\mathbb{R}^{N},H(\cdot))$ may lose the usual symmetry properties and the directional independence. The anisotropic Finsler $N$-Laplacian $\Delta ^{H}_{N}$ induced by $H$ is defined by
\begin{equation}\label{NL}
  \Delta ^{H}_{N}u:=\operatorname{div}\left(a(\nabla u)\right)=\operatorname{div}\left(H^{N-1}(\nabla u)\nabla H(\nabla u)\right)
\end{equation}
with
\begin{equation}\label{a}
  a(\xi):=H^{N-1}(\xi)\nabla H(\xi), \qquad \forall \,\, \xi\in\mathbb{R}^{N}.
\end{equation}
When $H(\cdot)$ equals to the Euclidean norm $|\cdot|$, the Finsler $N$-Laplacian $\Delta ^{|\cdot|}_{N}=\Delta_{N}$, where the regular $N$-Laplacian $\Delta_{N}$ is given by $\Delta_{N}u:=\operatorname{div}\left(|\nabla u|^{N-2}\nabla u\right)$ which arises as the first variation of the functional $\int|\nabla u|^N\mathrm{d}x$. Thus it is natural for us to consider the anisotropic Finsler $N$-Laplacian $\Delta ^{H}_{N}$ appearing in the Euler-Lagrange equations for Wulff-type functionals $\int H(\nabla u)^N\mathrm{d}x$. For literatures on anisotropic problems and their extensive applications, refer to e.g. \cite{ACF,BFK,BC,CRS,CFR,CL1,CL2,CFV,DP,DPV,FF,FM,Taylor,WX1,WX2,Wulff,XG,ZZ} and the references therein.

\medskip

We say $u$ is a solution of the anisotropic Finsler $N$-Laplacian Liouville equation \eqref{eq:1.1} with a finite mass  if $u$ is a function satisfying $u\in W^{1,N}_{loc}(\overline{\mathcal{C}})$ such that $\int_{\mathcal{C}}e^u\mathrm{d}x<+\infty$ and
$$ \int_{\mathcal{C}} H^{N-1}(\nabla u)\langle \nabla H(\nabla u), \nabla\psi\rangle \mathrm{d}x = \int_{\mathcal{C}}e^{u(x)}\psi(x)\mathrm{d}x$$
for all $\psi\in W^{1,N}_{0}(\Omega)\cap L^{\infty}(\Omega)$ with $\Omega\subset\mathbb{R}^{N}$ bounded.

\medskip

We will use the notation $H_0$ to denote the dual norm associated to $H$ defined by
\begin{equation}\label{def:2.1}
    H_0(x):=\sup\limits_{H(\xi)=1}\left\langle x,\xi\right\rangle, \qquad \forall \,\, x\in\mathbb{R}^N.
\end{equation}
In addition, we set
\begin{equation}\label{def:2.2}
    \widehat{H}_0(x):=H_0(-x), \qquad \forall \,\, x\in\mathbb{R}^N.
\end{equation}
Definitions \eqref{def:2.1} and \eqref{def:2.2} imply that both $H_0$ and $\widehat{H}_0$ are positively homogenous of degree $1$, i.e., $H_0(\lambda x)=\lambda H_0(x)$ and $\widehat{H}_0(\lambda x)=\lambda\widehat{H}_0(x)$ for any $\lambda>0$. Moreover, it follows from \eqref{norm} that, for any $x\neq0$,
\begin{equation*}
        H_0(x)=\sup\limits_{\xi\neq 0}\frac{\left\langle x,\xi\right\rangle }{H(\xi)}\geq \frac{\left\langle x,x\right\rangle }{H(x)}\geq \frac{1}{C_{H}}|x|,
\end{equation*}
consequently, we can deduce further that
\begin{equation}\label{norm0}
        \frac{1}{C_{H}}|x|\leq H_0(x), \, \widehat{H}_0(x)\leq |x|\sup\limits_{H(\xi)=1}|\xi|\leq \frac{1}{c_{H}}|x|, \qquad \forall \,\, x\in\mathbb{R}^{N}.
\end{equation}

\medskip

In \cite{CFR}, Ciraolo, Figalli and Roncoroni considered the following critical anisotropic $p$-Laplacian equations in convex cones $\mathcal{C}=\mathbb{R}^{k}\times \mathcal{\widetilde{C}}$:
\begin{equation}\label{eq:1.1p}
      \left\{
          \begin{aligned}
          &-\Delta ^{H}_{p}u=u^{p^{\ast}-1} \quad\,\,\, &{\rm{in}} \,\, \mathcal{C}, \\
          & u>0 \quad\,\,\, &{\rm{in}} \,\, \mathcal{C},\\
          &a(\nabla u)\cdot \nu =0 \quad\,\,\, &{\rm{on}} \,\, \partial\mathcal{C},\\
          &u\in D^{1,p}(\mathcal{C}),
          \end{aligned}
          \right.
    \end{equation}
where $1<p<N$, $\mathcal{\widetilde{C}}\subset\mathbb{R}^{N-k}$ is an open convex cone with vertex at the origin $0_{\mathbb{R}^{N-k}}$ which contains no lines, $\Delta ^{H}_{p}u:=\operatorname{div}\left(H^{p-1}(\nabla u)\nabla H(\nabla u)\right)$, $p^{\ast}:=\frac{Np}{N-p}$ and $D^{1,p}(\mathcal{C}):=\{u\in L^{p^{\ast}}(\mathcal{C})\mid \,\, \int_{\mathcal{C}}|\nabla u|^{p}\mathrm{d}x<+\infty\}$. Under the assumption on the gauge $H\in C_{loc}^2(\mathbb{R}^N\setminus \{0\})$ that
\begin{equation}\label{A}
  H^{2}\in C_{loc}^{1,1}(\mathbb{R}^{N}) \,\,\, \text{and it is uniformly convex},
\end{equation}
they proved that $u$ must be uniquely determined by $U^{H}$ up to translations and scalings, i.e.,
\begin{equation}
u(x)=U_{\lambda,x_0}^H(x):=\lambda^{\frac{N-p}{p}}U^{H}(\lambda(x-x_{0}))
\end{equation}
for some $\lambda>0$ and $x_0\in\overline{\mathcal{C}}$, where
\begin{equation}
  U^{H}(x):=\left(\frac{N^{\frac{1}{p}}\left(\frac{N-p}{p-1}\right)^{\frac{p-1}{p}}}{1+\widehat{H}_0\left(x\right)^{\frac{p}{p-1}}}\right)^{\frac{N-p}{p}}
\end{equation}
Moreover, \\
  $(\rm{i})$ if $k=N$, then $\mathcal{C}=\mathbb{R}^N$ and $x_0$ may be a generic point in $\mathbb{R}^N$;\\
  $(\rm{ii})$ if  $k\in\{1,\cdots ,N-1\}$, then $x_0\in\mathbb{R}^k\times\{0_{\mathbb{R}^{N-k}}\}$;\\
  $(\rm{iii})$ if $k=0$, then $x_0=0$.\\
In this paper, we will prove the classification result for the anisotropic Finsler $N$-Laplacian Liouville equation \eqref{eq:1.1} in convex cones with finite masses, which is a counterpart in the limiting case $p=N$ of the classification results in \cite{CFR} for the critical anisotropic $p$-Laplacian equations with $1<p<N$ in convex cones. For classification and other related results on $p$-Laplacian equations with $1<p<N$, c.f. \cite{DLL,DFSV,DMMS,D,S,S2,BS16,T,VJ16} and the references therein.

\medskip

Now we take the limiting case $p=N\geq2$ into account. When $p=N=2$, by using the method of moving planes, Chen and Li \cite{CL} classified all the $C^{2}$-smooth solutions with a finite total curvature of the following $2$-D Liouville equation:
\begin{equation}\label{0-1}\begin{cases}
-\Delta u(x)=e^{u(x)},  \qquad  x\in\mathbb{R}^{2}, \\
\int_{\mathbb{R}^{2}}e^{u(x)}\mathrm{d}x<+\infty.
\end{cases}\end{equation}
They proved that there exists some point $x_{0}\in\mathbb{R}^{2}$ and some $\lambda>0$ such that
$$u(x)=\ln\left[\frac{2\lambda}{1+\lambda^{2}|x-x_{0}|^{2}}\right].$$
The Liouville equation \eqref{0-1} was initially investigated by Liouville \cite{Liou}, which arises from a variety of situations, such as from prescribing Gaussian curvature in geometry and from combustion theory in physics. In \cite{CK,CW}, the authors reproved the classification results in \cite{CL} for $2$-D Liouville equation \eqref{0-1} via different approaches. For sphere covering inequality and its applications on Liouville type equations, please refer to \cite{GM} and see also \cite{BGJM,GHM,GJM,GL}. For classification of solutions to semi-linear equations on Heisenberg group or CR manifolds via Jerison-Lee identities and invariant tensor techniques, see \cite{MO,MOW} and the references therein. For classification and other related results on Liouville type systems or (higher order) Liouville type equations, c.f. \cite{BF,C,CY,CK,DQ2,Lin,WX} and the references therein.

\medskip

In higher dimensions $N\geq3$, the method of moving planes does not work for the $N$-Laplacian Liouville equations. By exploiting the isoperimetric inequality and Pohozaev identity, Esposito \cite{E} classified all solutions with finite mass to the following $N$-D Liouville equation:
\begin{equation}\label{0-1N}\begin{cases}
-\Delta_{N} u(x)=e^{u(x)},  \qquad  x\in\mathbb{R}^{N}, \\
\int_{\mathbb{R}^{N}}e^{u(x)}\mathrm{d}x<+\infty.
\end{cases}\end{equation}
He proved that there exists some point $x_{0}\in\mathbb{R}^{N}$ and some $\lambda>0$ such that
$$u(x)=\ln\left[\frac{c_{N}\lambda^{N}}{\left(1+\lambda^{\frac{N}{N-1}}|x-x_{0}|^{\frac{N}{N-1}}\right)^{N}}\right].$$
Subsequently, Ciraolo and Li \cite{CL2} derived classification results for the anisotropic $N$-Liouville equation in the whole space $\mathbb{R}^{N}$ and hence gave an affirmative answer to a conjecture made in \cite{WX2}. In this present paper, we will also extend the classification results in \cite{CK,CL,CW,CL2,E} for $N$-D quasi-linear Liouville equation in the whole space $\mathbb{R}^{N}$ to general convex cones $\mathcal{C}$ including $\mathbb{R}^{N}$, the half space $\mathbb{R}^{N}_{+}$ and $\frac{1}{2^{m}}$-space $\mathbb{R}^{N}_{2^{-m}}:=\{x\in\mathbb{R}^{N}\mid x_{1},\cdots,x_{m}>0\}$ ($m=1,\cdots,N$). For optimal Liouville type theorems on general bounded or unbounded domains (including all cone-like domains) and blowing-up analysis on not necessarily $C^1$-smooth domains, see \cite{DQ}.

\subsection{Main results}
In this paper, we shall first prove and apply the radial Poincar\'{e} type inequality (Lemma \ref{A1}) and the logarithmic asymptotic estimate for quasi-linear problem on convex cone $\mathcal{C}$ (Lemma \ref{le:1+}). Then, we exploit the anisotropic isoperimetric inequality inside convex cones (c.f. \cite{CRS}, see also \cite{CL1,DPV,LP}) and  approximation method in cones (c.f. \cite{CRS,CFR}, see also \cite{CL1}), and  the sharp asymptotic estimates on $u$ and $\nabla u$ at infinity in cones as well as  Pohozaev type identities.  All together allow us to  extend the classification results in \cite{CFR} from $1<p<N$ to the limiting case $p=N$ and derive the following complete classification theorem for the anisotropic Finsler $N$-Laplacian Liouville equation \eqref{eq:1.1} in convex cones.
\begin{thm}\label{Th:1.1}
Let $N\geq2$ and $\mathcal{C}$ be a convex cone in $\mathbb{R}^{N}$ and write $\mathcal{C}=\mathbb{R}^{k}\times\mathcal{\widetilde{C}}$, where $k\in\{0,\cdots ,N\}$ and $\mathcal{\widetilde{C}}\subset\mathbb{R}^{N-k}$ is an open convex cone with vertex at the origin $0_{\mathbb{R}^{N-k}}$ which does not contain a line. Assume that the gauge $H$ is uniformly elliptic, i.e., $B^{H}_{1}(0):=\{\xi\in\mathbb{R}^{N}\mid \, H(\xi)<1\}$ is uniformly convex. Let $u$ be a solution of \eqref{eq:1.1}. Then $u$ must be uniquely determined by $U^{H}$ up to translations and scalings, i.e.,
    \begin{equation}\label{eq:1.2}
      u(x)=U^{H}_{\lambda,x_{0}}(x):=U^{H}(\lambda(x-x_{0}))+N\ln\lambda
    \end{equation}
for some $\lambda>0$ and $x_0\in\overline{\mathcal{C}}$, where
    \begin{equation}
    U^{H}(x):=\ln\Bigg[\frac{c_N}{\Big(1+\widehat{H}_0^{\frac{N}{N-1}}(x)\Big)^N}\Bigg]
    \end{equation}
with $c_N:=N\left(\frac{N^2}{N-1}\right)^{N-1}$. Moreover, \\
  $(\rm{i})$ if $k=N$, then $\mathcal{C}=\mathbb{R}^N$ and $x_0$ may be a generic point in $\mathbb{R}^N$;\\
  $(\rm{ii})$ if  $k\in\{1,\cdots ,N-1\}$, then $x_0\in\mathbb{R}^k\times\{0_{\mathbb{R}^{N-k}}\}$;\\
  $(\rm{iii})$ if $k=0$, then $x_0=0$.
  \end{thm}

\begin{rem}
Our classification result in Theorem \ref{Th:1.1} is a counterpart in the limiting case $p=N$ of the classification results in \cite{CFR} for the critical anisotropic $p$-Laplacian equations with $1<p<N$ in convex cones, and also extends the classification results in \cite{CK,CL,CW,CL2,E} for $N$-D Liouville equation in the whole space $\mathbb{R}^{N}$ to general convex cones $\mathcal{C}$ including $\mathbb{R}^{N}$, the half space $\mathbb{R}^{N}_{+}$ and $\frac{1}{2^{m}}$-space $\mathbb{R}^{N}_{2^{-m}}$ ($1\leq m\leq N$).
\end{rem}
\begin{rem}
    One can easily verify that, the assumption ``$H$ is uniformly elliptic" in our Theorem \ref{Th:1.1} can be deduced from (and hence is weaker than) the condition ``$H^2$ is uniformly convex" assumed in \cite{CFR}. Indeed, if $H^2$ is uniformly convex, then there exist constants $0<\lambda\leq\Lambda$ such that
    $$\lambda \mathrm{Id}\leq H(\xi)D^2H(\xi)+\nabla H(\xi)\otimes \nabla H(\xi)\leq\Lambda\mathrm{Id}, \qquad \forall \,\, \xi\in\mathbb{R}^N\setminus\{0\},$$
which implies
$$\sum\limits_{i,j=1}^{N}H_{ij}(\xi)\zeta_i\zeta_j\geq\lambda|\zeta|^2, \qquad \forall \,\, \xi\in\partial B^H_1(0),\ \zeta\in\nabla H(\xi)^\perp ,$$
and thus $H$ is uniformly elliptic (see \cite{CFV}).
\end{rem}

We would like to mention some key ingredients and the outline of the proof of our main Theorem \ref{Th:1.1}.

\smallskip

First, we note that the classical methods such as the method of moving planes/spheres in conjunction with Kelvin type transforms for the higher order Liouville equations involving $(-\Delta)^{\frac{N}{2}}$ in \cite{CY,CK,CL,CW,DQ2,Lin,WX},  the method of moving planes for the critical $p$-Laplacian equations with $1<p<N$ in \cite{DLL,DFSV,DMMS,D,S,S2,BS16,T,VJ16} and the approach in \cite{CFR} for the critical $p$-Laplacian equations in convex cones no longer work for the $N$-Laplacian Liouville equation. We shall gain the classification result via isoperimetric inequalities (c.f. e.g. \cite{CL,CL2,E,KP}).

\smallskip

In order to get the sharp asymptotic estimates on $u$ and $\nabla u$ at infinity in cones, we need to prove and apply the radial Poincar\'{e} type inequality and the logarithmic asymptotic estimate for quasi-linear problem on convex cone $\mathcal{C}$, and use the Brezis-Merle type exponential inequality for Finsler $N$-Laplacian, Serrin's local $L^{\infty}$ estimate and the limiting arguments from \cite{CL2}. Finally, we can prove that $u(x)+\beta\log \widehat{H}_0(x)\in L^{\infty}(\overline{\mathcal{C}}\setminus B^{\widehat{H}_{0}}_{R_0}(0))$ (see Section 3). We would like to point out that due to the Finsler setting of $\Delta^{H}_{N}$, the approaches such as  Kelvin type transforms and the Green representation formulae in \cite{CK,CL,E}, the scaling argument and doubling argument (c.f. \cite{PQS}) in \cite{BS16,VJ16} and the De Giorgi-Moser-Nash iteration in \cite{CPY,DLL} do not seem to work.

\smallskip

One should note that, since we are dealing with problems in cones, in order to treat the boundary condition $\left\langle a(\nabla u),\nu\right\rangle=0$ on $\partial\mathcal{C}$, we need to prove and apply the radial Poincar\'{e} type inequality (see Lemma \ref{A1}) instead of the usual Sobolev embedding $W^{1,p}_{0}(\Omega)\hookrightarrow L^{\frac{Np}{N-p}}(\Omega)$ with $1\leq p<N$, which has its own importance and interests. The radial Poincar\'{e} type inequality only require the function $f\in W^{1,p}(\Omega)$ to satisfy $f=0$ on the (back) radial contact set $\Gamma^{+}_{P}\subset\partial\Omega$, where $P$ is the radial center (see Appendix \ref{appendix}). In particular, if $\Omega=\mathcal{C}\cap E$ with $E\subset\mathbb{R}^{N}$ be a bounded domain, then $P=0$ and $\Gamma^{+}_{0}\subset\partial E\cap\mathcal{C}$. We note that this inequality yields the usually well-known Poincar\'{e} inequality as a direct corollary and has many interesting applications (see Appendix \ref{appendix}).

\smallskip

Next, since we are dealing with problems in cones, the second order regularity up to the boundary should be taken into account carefully. The main difficulties are the degeneracy of the equation and the non-smoothness of $\partial\mathcal{C}$. We will prove the second order regularity $a(\nabla u)\in W^{1,2}_{loc}(\overline{\mathcal{C}})$ for weak solution $u$ via an approximation method in cones appeared in \cite{CFR} and also in \cite{CL1} (see Section 4). Since $u$ is not $C^1$-smooth up to the boundary $\partial\mathcal{C}$, we need to apply an approximation argument used in the proof of \cite[Theorem 1.3]{CRS} to determine the precise value of the constant $\beta$ in the sharp asymptotic estimates on $u$ and $\nabla u$. Then, based on the sharp asymptotic estimates on $u$ and $\nabla u$ and the second order regularity, we can prove a delicate quantization result on the total mass $\int_{\mathcal{C}}e^u\mathrm{d}x$ by using the anisotropic isoperimetric inequality inside convex cones and the Pohozaev type identity, which indicates the equality holds in the isoperimetric inequality inside convex cones (see Section 5). Finally, from the characteristic of the equality case (see Theorem 2.5), we conclude that
  $$\partial\mathcal{C}_t:=\{x\in\mathcal{C}\mid\,u(x)=t\}=\partial B^{\widehat{H}_0}_{R(t)}(x(t))\cap\mathcal{C}$$
 and
 $$ e^t=\frac{c_N\lambda^N}{\left[1+\lambda^{\frac{N}{N-1}}R^{\frac{N}{N-1}}(t)\right]^N},$$
where
\begin{equation*}
    \lambda:=\left[\frac{e^{t_0}}{N\left(\frac{N^2}{N-1}\right)^{N-1} }\right]^{\frac{1}{N}},\quad t=u(x),\quad \widehat{H}_0(x-x(t))=R(t),\quad x\in\partial\mathcal{C}_{t}.
\end{equation*}Moreover, by an approximation method, we can show that, at $x_{1}=x(t_{1})$ with $x'(t_{1})\neq0$,
 $$1=-C(\widehat{H}_0(x_1)+R^{\prime}(t_1))=-C\left(-H_0(x_1)+R^{\prime}(t_1)\right),$$
 which is absurd, and hence $x'(t)\equiv0$, that is, $x(t)\equiv x_0\in \mathbb{R}^k\times\{0_{\mathbb{R}^{N-k}}\}$. The classification results in Theorem \ref{Th:1.1} follow immediately (see Section 6).

\medskip

The rest of our paper is organized as follows. In Section 2, we will give some preliminaries on some basic properties of the gauge $H$, the anisotropic isoperimetric inequality inside convex cones, the Brezis-Merle type exponential inequality for Finsler $N$-Laplacian, Serrin's local $L^{\infty}$ estimate, the logarithmic asymptotic estimate for quasi-linear problem on convex cone $\mathcal{C}$, the comparison principle and Liouville type result in convex cones, property of the trace of matrices, and so on. Section 3 and Section 4 are devoted to proving sharp asymptotic estimates at infinity on $u$ and $\nabla u$, and the second-order regularity for weak solutions, respectively. In Section 5, we derive precise quantization results on the total mass $\int_{\mathcal{C}}e^{u}\mathrm{d}x$ by the anisotropic isoperimetric inequality inside convex cones and the Pohozaev type identity. Finally, we will carry out our proof of Theorem \ref{Th:1.1} in Section 6. In appendix \ref{appendix}, we will prove the radial Poincar\'{e} type inequality (Lemma \ref{A1}) and apply it to show the Brezis-Merle type exponential inequality for mixed boundary value problems of Finsler $N$-Laplacian in Lemma \ref{le:1+a} and Remark \ref{rem-a}, which have their own importance and interests.

\medskip

In the following, we will use $C$ to denote a general positive constant that may depend on $N$, $k$, $\mathcal{C}$, $H$ and $u$, and whose value may differ from line to line.

\section{Preliminaries}
We need the following two Lemmas on basic properties of the gauge $H$.
\begin{lem}\cite[Lemma 2.2]{CL1}\label{8,le:2.2}
If the gauge $H\in C^2(\mathbb{R}^N\setminus\{0\})$ and the Hessian of $H^{N}$ is positive definite in $\mathbb{R}^N\setminus\{0\}$, then $H$ is convex and $H_{0}\in C^2(\mathbb{R}^N\setminus\{0\})$. Moreover, for $x,\xi\in\mathbb{R}^N\setminus\{0\}$,
    \begin{equation}
H\left(\nabla H_0(x)\right)=H_0(\nabla H(\xi))=1,
\end{equation}
and
\begin{equation}\label{eq:2.2}
x=H_0(x) \nabla H\left(\nabla H_0(x)\right), \quad \xi=H(\xi) \nabla H_0(\nabla H(\xi)),
\end{equation}
where $H_0$ is the dual function of $H$ defined by \eqref{def:2.1}.
\end{lem}
\begin{lem}\cite[Lemma 2.3]{CL1}\label{8,le:2.3}
    Assume that $H\in C^2(\mathbb{R}^N\setminus\{0\})$ is uniformly elliptic. Then there exists $\lambda>0$ such that
    \begin{equation}\label{ineq:2.3}
        \partial_{ij}^2(H^N)(\xi)\eta_i\eta_j\geq\frac{1}{\lambda}|\xi|^{N-2}|\eta|^2\qquad
        \text{and}\qquad \sum\limits_{i,j}| \partial_{ij}^2(H^N)(\xi)|\leq\lambda|\xi|^{N-2}
    \end{equation}
    for any $\xi\in\mathbb{R}^N\setminus\{0\}$ and $\eta\in\mathbb{R}^N$. Furthermore, there exist $c_1$, $c_2>0$, depending only on $N$ and $\lambda$, such that
    \begin{equation}\label{aeq:2.4}
        \begin{aligned}
            \langle a(\xi_1)-a(\xi_2), \xi_1-\xi_2\rangle\geq c_1(|\xi_1|+|\xi_2|)^{N-2}|\xi_1-\xi_2|^2
        \end{aligned}
    \end{equation}
    and
    \begin{equation}\label{aeq:2.5}
        \begin{aligned}
            |a(\xi_1)-a(\xi_2)|\leq c_2(|\xi_1|+|\xi_2|)^{N-2}|\xi_1-\xi_2|
        \end{aligned}
    \end{equation}
    for any $\xi_1$, $\xi_2\in\mathbb{R}^N\setminus\{0\}$.
\end{lem}

We also need the following two lemmas on Brezis-Merle type exponential inequality (c.f. \cite{BF}) for Finsler $N$-Laplacian (see \cite{AP,E,RW} for $N$-Laplacian and \cite{CL2,WX2,XG} for Finsler $N$-Laplacian) and Serrin's local $L^{\infty}$ estimate (c.f. \cite{S}, see also \cite{CL2}), respectively.
\begin{lem}\cite[Lemma 2.2]{CL2}\label{le:1}
    Let $\Omega$ be a bounded domain in $\mathbb{R}^N$ and $f\in L^1(\Omega)$. Let u $\in W^{1,N}(\Omega)$ be a weak solution of
    \begin{equation}\label{eq:2.1}
        -\Delta ^{H}_{N}u=f\ \ \text{\rm in}\ \Omega
    \end{equation}
    and let $h\in W^{1,N}(\Omega)$ be the weak solution of
    \begin{equation}
        \left\{
            \begin{aligned}
            &-\Delta ^{H}_{N}h=0 \qquad &\rm{in}\,\, \Omega, \\
            &h =u\qquad &\rm{on}\,\, \partial\Omega.
            \end{aligned}
            \right.
    \end{equation}
    Then there exists a constant $\mu$, depending only on $N$ and $H$, such that for every
    $0<\lambda<\mu \|f\|_{L^1(\Omega)}^{-\frac{1}{N-1}}$, it holds
    \begin{equation*}
        \int_{\Omega}e^{\lambda|u-h|}\mathrm{d}x\leq \frac{\mathcal{L}^N(\Omega)}{1-\lambda\mu^{-1}\|f\|_{L^1(\Omega)}^{\frac{1}{N-1}}}.
    \end{equation*}
\end{lem}
\begin{lem}\cite[Lemma 2.3]{CL2}\label{le:2}
    Let $u\in W_{loc}^{1,N}(\Omega)$ be a weak solution of \eqref{eq:2.1} with $f\in L^{\frac{N}{N-\varepsilon}}(\Omega)$ for
    some $0 < \varepsilon \leq 1$. Assume $B_{2R}\subset\Omega$. Then
    \begin{equation*}
       \|u^+\|_{L^{\infty}(B_R)}
       \leq CR^{-1}\|u^+\|_{L^{N}(B_{2R})}+CR^{\frac{\varepsilon}{N-1}}\|f\|_{L^{\frac{N}{N-\varepsilon}}(B_{2R})}^{\frac{1}{N-1}},
    \end{equation*}
    where $C=C(N, \varepsilon, H)$ is a constant.
\end{lem}

Consider the following quasi-linear problem on convex cone $\mathcal{C}$:
\begin{equation}\label{eq-2}
\left\{
          \begin{aligned}
          &\operatorname{div}\,{\bf{a}}(\nabla u)=f(x)\quad\,\,\, &\text{in} \,\, \mathcal{C}\setminus\overline{B_1(0)}, \\
          &{\bf{a}}(\nabla u)\cdot \nu =0 \quad\,\,\, &{\rm{on}} \,\, \partial\mathcal{C}\setminus\overline{B_1(0)},
          \end{aligned}
          \right.
\end{equation}
where ${\bf{a}}:\,\mathbb{R}^{N}\rightarrow\mathbb{R}^{N}$ satisfies
\begin{equation}\label{eq-3}
	|{\bf{a}}(p)|\leq a_{1}\,|p|^{N-1}\quad\,\,\text{and}\quad\,\,\left<{\bf{a}}(p),p\right>\geq a_{0}\,|p|^N, \qquad\forall \,\, p\in\mathbb{R}^N
\end{equation}
for some constants $a_{1},a_{0}>0$, and $f\in L^1(\mathcal{C}\setminus\overline{B_1(0)})$ is a scalar function satisfying
\begin{equation}\label{eq-4}
	0\le f(x)\le\frac{C}{|x|^N}, \quad\,\,\,\,\, \text{a.e.} \,\, x\in\mathcal{C}\setminus\overline{B_1(0)}
\end{equation}
for some constant $C>0$. It follows from the elliptic estimates in \cite{D,S2,T} and the boundary regularity in \cite{L} that $u\in C^{1,\theta}_{loc}\left(\overline{\mathcal{C}}\setminus\overline{B_1(0)}\right)$ for some $\theta\in(0,1)$. We can derive the following logarithmic asymptotic estimate in convex cone $\mathcal{C}$ at infinity by using the radial Poincar\'{e} type inequality in Lemma \ref{A1} in Appendix \ref{appendix}, which has its own interests and importance.
\begin{thm}[Logarithmic asymptotic behavior in convex cone $\mathcal{C}$ at infinity]\label{le:1+}
Suppose \eqref{eq-3} and \eqref{eq-4} hold. Let $u$ be a nonnegative solution of \eqref{eq-2}. If $u(x)\to+\infty$ as $|x|\to+\infty$ with $x\in\mathcal{C}$, then there exists a constant $L>1$ such that
	\begin{equation}\label{eq-5}
		\frac{1}{L}\log|x|\le u(x)\le L\log|x|
	\end{equation}
for any $x\in\mathcal{C}$ with $|x|$ large enough.
\end{thm}
\begin{proof}
First, in the particular homogeneous case $f\equiv0$, the logarithmic asymptotic estimate in \eqref{eq-5} can be deduced from (36) in \cite[Theorem 2]{S2} with $\beta=N$ in a similar way as the lower bound in (2.8) of \cite[Proposition 2.3]{CFR} (c.f. also pages 159-160 in \cite{VJ16}). In fact, due to the boundary condition $\left\langle{\bf{a}}(\nabla u),\nu\right\rangle=0$ on $\partial\mathcal{C}\setminus\overline{B_1(0)}$, by restricting the integration region to the convex cone $\mathcal{C}$ and changing $\omega_{N}=\mathcal{H}^{N-1}(\partial B_{1}(0))$ to $\omega_{N}^{\mathcal{C}}:=\mathcal{H}^{N-1}(\partial B_{1}(0)\cap\mathcal{C})$, Lemmas 8, 9 and 12 in \cite{S2} are still valid in convex cone $\mathcal{C}$, up to slightly corresponding modifications. Furthermore, the estimates (30) and (35) in \cite{S2} with $\beta=N$ still hold in $\mathcal{C}$, i.e., $m(\sigma):=\min\limits_{|x|=\sigma,\,x\in\mathcal{C}}u(x)\leq C_{1}\log |x|$ and $M(\sigma):=\max\limits_{|x|=\sigma,\,x\in\mathcal{C}}u(x)\geq C_{0}\log |x|$ for $\sigma$ large enough and some constant $C_{0},\,C_{1}>0$ independent of $\sigma$. Thus it follows from the Harnack inequality (c.f. e.g. \cite{S,S2}) that the logarithmic asymptotic behavior (36) in \cite[Theorem 2]{S2} with $\beta=N$ indeed holds for $u$ in $\mathcal{C}$, i.e., \eqref{eq-5} holds.

\smallskip

One should note that, in order to deal with the boundary condition $\left\langle{\bf{a}}(\nabla u),\nu\right\rangle=0$ on $\partial\mathcal{C}\setminus\overline{B_1(0)}$, in the proof of Harnack inequality in \cite{S}, the Sobolev embedding $W^{1,p}_{0}(\Omega)\hookrightarrow L^{\frac{Np}{N-p}}(\Omega)$ with $1\leq p<N$ in \cite[Lemma 5]{S} should be replaced by the radial Poincar\'{e} type inequality in Lemma \ref{A1} in Appendix \ref{appendix}, which has its own interests and importance. Indeed, the radial Poincar\'{e} type inequality in Lemma \ref{A1} implies that, for any $f\in W^{1,p}(\Omega)$ with $f=0$ on the (back) radial contact set $\Gamma^{+}_{P}\subset\partial\Omega$,
\begin{equation}\label{eq-6}
  \|f\|_{L^{\frac{Np}{N-p}}(\Omega)}\leq S(\Omega)\left(\|\nabla f\|_{L^{p}(\Omega)}+\|f\|_{L^{p}(\Omega)}\right)\leq S(\Omega)\left(1+d_{r,P}(\Omega)\right)\|\nabla f\|_{L^{p}(\Omega)},
\end{equation}
where $S(\Omega)$ is the Sobolev constant of the embedding $W^{1,p}(\Omega)\hookrightarrow L^{\frac{Np}{N-p}}(\Omega)$, and $d_{r,P}(\Omega)$ defined by \eqref{rw} denotes the radial width of $\Omega$ with respect to the radial center $P$. In particular, if $\Omega=\mathcal{C}\cap E$ with $E\subset\mathbb{R}^{N}$ be a bounded domain, then $P=0$ and $\Gamma^{+}_{0}\subset\partial E\cap\mathcal{C}$.

\smallskip

Then, we can extend the logarithmic asymptotic estimate \eqref{eq-5} in convex cone $\mathcal{C}$ from the homogeneous case $f\equiv0$ to the inhomogeneous case $f\not\equiv0$ satisfying \eqref{eq-4}, by using similar way as in the proof of Theorem 6.1 in \cite{CL2}. This completes the proof of Theorem \ref{le:1+}.
\end{proof}

\smallskip

For any open subset $D\subset\mathbb{R}^{N}$ and measurable set $E\subset\mathbb{R}^{N}$, the anisotropic perimeter of $E$ in $D$ with respect to the gauge $H$ is defined by
\begin{equation}
P_H(E;D)=\sup \left\{\int_E \operatorname{div} \Phi\,\mathrm{d} x\mid\, \Phi \in C_c^1\left(D ; \mathbb{R}^N\right), H_0(\Phi) \leq 1\right\}.
\end{equation}
In particular, if $E$ has locally finite perimeter, one has
\begin{equation}
P_H(E ; D)=\int_{D \cap \partial^* E} H(\nu) \mathrm{d}\mathcal{H}^{N-1},
\end{equation}
where $\partial^* E$ is the reduced boundary of $E$ and $\nu$ is the outer (measure theoretic) unit normal vector to $E$ defined on $\partial^* E$.

\smallskip

The anisotropic isoperimetric inequality inside convex cones was proved in \cite[Theorem 1.3]{CRS}. For the characterization of the equality case, refer to \cite[Theorem 4.2]{DPV} and \cite[Theorem 2.2]{FI} when $H$ is a norm, and \cite[Theorem 2.5]{CL1} when $H$ is a gauge. For (sharp) quantitative isoperimetric inequalities, see  e.g. \cite{FMP1,FMP0}.
\begin{thm}\label{thm1}
Let $\mathcal{C}=\mathbb{R}^k \times \widetilde{\mathcal{C}}$, where $k \in\{0, \ldots, N\}$ and $ \widetilde{\mathcal{C}} \subset \mathbb{R}^{N-k}$ is an open, convex
cone with vertex at the origin $0_{\mathbb{R}^{N-k}}$ which contains no lines. Let $H$ be a gauge in $\mathbb{R}^{N}$, which is positive on $\mathbb{S}^{N-1}$ and let $H_0$ be its dual function defined by \eqref{def:2.1}. Then for each measurable set $E\subset\mathbb{R}^{N}$ with $\mathcal{H}^N(\mathcal{C}\cap E)<+\infty$,
\begin{equation}
\frac{P_H(E;\mathcal{C})}{\mathcal{H}^N(\mathcal{C}\cap E)^{\frac{N-1}{N}}} \geq \frac{P_H\left(B_1^{H_0}; \mathcal{C}\right)}{\mathcal{H}^N\left(\mathcal{C} \cap B_1^{H_0}\right)^{\frac{N-1}{N}}}.
\end{equation}
Moreover, the equality sign holds if and only if $\mathcal{C} \cap E=\mathcal{C} \cap B_R^{H_0}\left(x_0\right)$ for some $R>0$ and $x_0 \in \mathbb{R}^k \times\left\{0_{\mathbb{R}^{N-k}}\right\}$, where the Wulff ball $B^{H_{0}}_{R}(x_{0}):=\{x\in\mathbb{R}^{N}\mid \, H_{0}(x-x_{0})<R\}$.
\end{thm}

We also need the following generalized version of the divergence theorem from \cite{CL1}.
\begin{lem}\cite[Lemma 4.3]{CL1}\label{le:2.3}
   Let $\Omega$ be a bounded open subset of $\mathbb{R}^N$ with Lipschitz boundary and let $f\in L^1(\Omega)$. Assume that
   ${\bf{a}} \in C^0(\overline{\Omega};\mathbb{R}^N)$ satisfies $\operatorname{div} \, {\bf{a}}=f$ in the sense of distributions
   in $\Omega$. Then we have
   \begin{equation*}
    \int_{\partial\Omega}\left\langle {\bf{a}},\nu\right\rangle \mathrm{d}\mathcal{H}^{N-1}=\int_{\Omega}f(x)\mathrm{d}x.
   \end{equation*}
\end{lem}
The following two Lemmas on the comparison principle and Liouville type result in convex cones $\mathcal{C}$ will be used.
\begin{lem}\cite[Lemma 2.4]{CL1}\label{le:2.5}
    Let $\mathcal{C}\subset\mathbb{R}^{N}$ be an open, convex cone and $E\subset\mathbb{R}^{N}$ be a bounded domain such that $\mathcal{H}^{N-1}(\Gamma_0)>0$ where $\Gamma_0:=\mathcal{C}\cap \partial E$,  and assume  $\mathcal{C}\cap E$ is connected. Let the gauge $H$ be the same as in Lemma \ref{8,le:2.3}. Assume that $u,v\in W^{1,N}(\mathcal{C}\cap E)\cap C^{0}((\mathcal{C}\cap E)\cup \Gamma_0)$ satisfy
    \begin{equation*}
     \left\{
         \begin{aligned}
             &-\Delta ^{H}_{N}u\leq-\Delta ^{H}_{N}v \ &\rm{in}\,\, \mathcal{C}\cap E, \\
             &u\leq v\ &\rm{on}\,\, \Gamma_0,\\
             &\langle a(\nabla u),\nu\rangle=\langle a(\nabla v),\nu\rangle=0 \quad\,\, &\rm{on}\,\, \partial(\mathcal{C}\cap E)\setminus\Gamma_0.
             \end{aligned}
             \right.
    \end{equation*}
    Then $u\leq v$ in $\mathcal{C}\cap E$.
 \end{lem}
\begin{lem}\cite[Lemma 3.2]{CL1}\label{le:2.8}
    Let $\mathcal{C}$, $H$ and $H_0$ be the same as in Theorem \ref{Th:1.1}. Let $\gamma\in\mathbb{R}$ be a constant. Assume that $G(x)\in W^{1,N}_{loc}(\overline{\mathcal{C}}\setminus\{0\})\cap L^{\infty}(\mathcal{C})$ and the function $\gamma\log \widehat{H}_0(x)+G(x)$ satisfies
    \begin{equation*}
        \left\{
            \begin{aligned}
                &\Delta ^{\widehat{H}}_{N}(\gamma\log \widehat{H}_0(x)+G(x))=0 \ &\rm{in}\ \mathcal{C}, \\
                &\langle a(\nabla(\gamma\log \widehat{H}_0(x)+G(x))),\nu\rangle=0\ &\rm{on}\ \partial\mathcal{C}\setminus\{0\}.
                \end{aligned}
                \right.
       \end{equation*}
       Then $G(x)$ is a constant function.
\end{lem}
The following Lemma on a property of the trace of matrices will also be needed.
\begin{lem}\cite[Lemma 4.5]{AKM}\label{le:2.9}
    Let the matrix $A$ be symmetric with positive eigenvalues, and let $\lambda_{min}$ and $\lambda_{max}$ be its smallest and largest eigenvalue, respectively;
     let $B$ be a symmetric matrix. Then we have
     \begin{equation}
        \mathrm{trace}((AB)^2)\approx  \mathrm{trace}(AB(AB)^T),
     \end{equation}
     where the involved constants depend only on the ratio $\frac{\lambda_{max}}{\lambda_{min}}$ and $N$. In particular, we have
     \begin{equation}
        \mathrm{trace}(AB(AB)^T)\leq N\left(\frac{\lambda_{max}}{\lambda_{min}}\right) ^2  \mathrm{trace}((AB)^2).
     \end{equation}
\end{lem}

\section{Sharp asymptotic estimates at infinity on $u$ and $\nabla u$}
In this section, we will prove sharp asymptotic estimates on both $u$ and $\nabla u$ at infinity in cones. To this end, we will first show the following upper bound on $u$.
\begin{prop}\label{prop:3.1}
    Let $\mathcal{C}$ be a convex cone and $u$ be a solution of
    \begin{equation}\label{eq:3.1}
        \left\{
        \begin{aligned}
        &-\Delta ^{H}_{N}u=e^u \qquad &\rm{in}\,\, \mathcal{C}, \\
        &a(\nabla u)\cdot \nu =0\qquad &\rm{on}\,\, \partial\mathcal{C},\\
        &\int_{\mathcal{C}}e^u\mathrm{d}x<+\infty.
        \end{aligned}
        \right.
    \end{equation}
Then $u^{+}\in L^{\infty}(\mathcal{C})$ and $u\in C^{1,\theta}(\mathcal{C})$ for some $\theta\in(0,1)$. Moreover, there exists a constant $C>0$
such that
\begin{equation}\label{prop:3.2}
    u(x)\leq C-N\log|x|,\qquad \forall \,\,  x\in\mathcal{C}.
\end{equation}
\end{prop}
\begin{proof}
 For $\overline{x}\in\mathcal{C}$ and $0<r<1$, let $h\in W^{1,N}(B_r(\overline{x})\cap\mathcal{C})$ be the weak
 solution of the equation
 \begin{equation*}
    \left\{
        \begin{aligned}
        &-\Delta ^{H}_{N}h=0 \qquad &{\rm in}\,\, B_r(\overline{x})\cap\mathcal{C}, \\
        &h=u\qquad &{\rm{on}} \,\, \partial(B_r(\overline{x})\cap\mathcal{C}).
        \end{aligned}
        \right.
 \end{equation*}
 From comparison principle, we get $h\leq u$ in $B_r(\overline{x})\cap\mathcal{C}$. Then we have
 \begin{equation}\label{ineq:3.3}
    \int_{B_r(\overline{x})\cap\mathcal{C}}(h^+)^N\mathrm{d}x\leq \int_{B_r(\overline{x})\cap\mathcal{C}}(u^+)^N\mathrm{d}x
    \leq N!\int_{\mathcal{C}}e^u\mathrm{d}x.
 \end{equation}
 Applying Lemma \ref{le:2} to $h^+$, we deduce that
 \begin{equation}\label{ineq:3.4}
    \|h^+\|_{L^{\infty}\left(B_{\frac{r}{2}}(\overline{x})\cap\mathcal{C}\right)}\leq C(r)
 \end{equation}
 for some constant $C(r)>0$. Now, let $r>0$ small enough such that
 \begin{equation}\label{ineq:3.5}
    \left(\int_{B_{r}(\overline{x})\cap\mathcal{C}}e^u\mathrm{d}x\right)^{\frac{1}{N-1}}\leq \frac{\mu(N-1)}{2N},
 \end{equation}
 where $\mu$ is the same as in Lemma \ref{le:1}. Choosing $\lambda=\frac{N}{N-1}$ in Lemma \ref{le:1}, we obtain that
 \begin{equation}\label{ineq:3.6}
    \int_{B_{r}(\overline{x})\cap\mathcal{C}}e^{\frac{N}{N-1}(u-h)}\mathrm{d}x\leq C(r).
 \end{equation}
 Then, by \eqref{ineq:3.4} and \eqref{ineq:3.6}, we have
 \begin{equation}\label{ineq:3.7}
    \int_{B_{\frac{r}{2}}(\overline{x})\cap\mathcal{C}}e^{\frac{N}{N-1}u}\mathrm{d}x=\int_{B_{\frac{r}{2}}(\overline{x})\cap\mathcal{C}}e^{\frac{N}{N-1}(u-h)+\frac{N}{N-1}h}\mathrm{d}x
    \leq C(r)e^{C(r)}.
 \end{equation}
Using Lemma \ref{le:2} in conjunction with \eqref{ineq:3.3} and \eqref{ineq:3.7}, we obtain
\begin{equation}\label{ineq:3.8}
    \|u^+\|_{L^{\infty}\left(B_{\frac{r}{4}}(\overline{x})\cap\mathcal{C}\right)}\leq C(r).
\end{equation}
Since $\int_{\mathcal{C}}e^u\mathrm{d}x<+\infty$, there exists a $R>0$ large enough such that
\begin{equation}\label{ineq:3.9}
    \int_{\mathcal{C}\setminus B_{R}(0)}e^u\mathrm{d}x\leq\left[\frac{\mu(N-1)}{2N}\right]^{N-1}.
\end{equation}

For every $\overline{x}\in\mathcal{C}\setminus \overline{B_{R+1}(0)}$, we know that $B_1(\overline{x})\cap\mathcal{C}\subset\mathcal{C}\setminus B_{R}(0)$. By \eqref{ineq:3.9}, we get the validity of \eqref{ineq:3.5} with $r=1$. Hence, by Lemmas \ref{le:1} and \ref{le:2}, we have
\begin{equation}\label{ineq:3.10}
    \|u^+\|_{L^{\infty}(\mathcal{C}\setminus \overline{B_{R+1}(0)})}\leq C(1).
\end{equation}
By the compactness of $\overline{C\cap B_{R+1}}$, there exist a finite number of points $x_i$ and $r_{x_i}>0$ small enough
$(i=1,\cdots, L)$ fulfilling \eqref{ineq:3.5} so that
\begin{equation}
    \overline{C\cap B_{R+1}}\subset\bigcup_{i=1}^{L}B_{\frac{r_{x_i}}{4}}(x_i).
\end{equation}
Thus by \eqref{ineq:3.8}, we have
\begin{equation}\label{ineq:3.12}
    \|u^+\|_{L^{\infty}(\mathcal{C}\cap\overline{B_{R+1}})}\leq \max_{1\leq i\leq L}C(r_{x_i})<+\infty.
\end{equation}
Combining \eqref{ineq:3.10} with \eqref{ineq:3.12}, we have $u^+\in L^{\infty}(\mathcal{C})$, and hence $e^u\in L^{\infty}(\mathcal{C})$. Since $u\in L^N_{loc}(\mathcal{C})$, by the interior regularity results in \cite{D,S2,T}, we
 deduce that $u\in C^{1,\theta}_{loc}(\mathcal{C})$ for some $\theta\in(0,1)$.

\smallskip

Now we prove the property \eqref{prop:3.2}. For $R>0$ and $x\in\mathcal{C}$, let $u_R(x)=u(Rx)+N\log R$. Note that the convex cone $\mathcal{C}$ satisfies $R\mathcal{C}=\mathcal{C}$ for any $R>0$, we know that $u_R$ is well defined on $\mathcal{C}$. Since $u$ satisfies the equation \eqref{eq:3.1}, it follows that $u_R$ is also a solution of \eqref{eq:3.1}. Given $x_0\in \mathbb{S}^{N-1}\cap\mathcal{C}$,
 let $h_R\in W^{1,N}(B_{\frac{1}{2}}(x_0)\cap\mathcal{C})$ be the weak solution of
 \begin{equation*}
    \left\{
        \begin{aligned}
        &-\Delta ^{H}_{N}h_R=0 \qquad &{\rm in}\,\, B_{\frac{1}{2}}(x_0)\cap\mathcal{C}, \\
        &h_R=u_R\qquad &{\rm{on}}\,\, \partial\left(B_{\frac{1}{2}}(x_0)\cap\mathcal{C}\right).
        \end{aligned}
        \right.
 \end{equation*}
Similar to \eqref{ineq:3.3}, we obtain that
 \begin{equation*}
    \int_{B_{\frac{1}{2}}(x_0)\cap\mathcal{C}}(h_R^+)^N\mathrm{d}x\leq \int_{B_{\frac{1}{2}}(x_0)\cap\mathcal{C}}(u_R^+)^N\mathrm{d}x
    \leq N!\int_{\mathcal{C}}e^{u_R}\mathrm{d}x=N!\int_{\mathcal{C}}e^{u}\mathrm{d}x.
 \end{equation*}
 Hence by Lemma \ref{le:2}, there exists a positive constant $C$ independent of $R$ and $x_0$ such that
 \begin{equation*}
    \|h^+_R\|_{L^{\infty}\left(B_{\frac{1}{2}}(x_0)\cap\mathcal{C}\right)}\leq C.
 \end{equation*}
  \par Since
  \begin{equation*}
    \int_{B_{\frac{1}{2}}(x_0)\cap\mathcal{C}}e^{u_R}\mathrm{d}x\leq \int_{\mathcal{C}\setminus B_{\frac{1}{2}}(0)}e^{u_R}\mathrm{d}x
    \leq \int_{\mathcal{C}\setminus B_{\frac{R}{2}}(0)}e^{u}\mathrm{d}x,
  \end{equation*}
  we can choose $R_0$ large enough such that for any $R>R_0$, there holds
  \begin{equation*}
    \int_{B_{\frac{1}{2}}(x_0)\cap\mathcal{C}}e^{u_R}\mathrm{d}x\leq \int_{\mathcal{C}\setminus B_{\frac{R}{2}}(0)}e^{u}\mathrm{d}x\leq \left[\frac{\mu(N-1)}{2N}\right]^{N-1}.
  \end{equation*}
  Then, by the argument in \eqref{ineq:3.6}-\eqref{ineq:3.8}, we obtain
  \begin{equation*}
    \|u_R^+\|_{L^{\infty}\left(B_{\frac{1}{8}}(x_0)\cap\mathcal{C}\right)}\leq C
  \end{equation*}
  for any $R>R_0$ and constant $C$ independent of $R$ and of $x_0$. Noting that $\mathbb{S}^{N-1}\cap\mathcal{C}$ can be covered by a finite number of sets
   $B_{\frac{1}{8}}(x_0)\cap \mathcal{C}$, $x_0\in\mathbb{S}^{N-1}$, thus we get
   \begin{equation}\label{ineq:3.13}
    \|u^+_R\|_{L^{\infty}(\mathbb{S}^{N-1}\cap \mathcal{C})}\leq C.
   \end{equation}
By \eqref{ineq:3.13}, for all $x\in\mathcal{C}\setminus \overline{B_{R_0}(0)}$, we have
\begin{equation}\label{ineq:3.14}
    u(x)+N\log|x|\leq C.
\end{equation}
Since \eqref{ineq:3.14} is valid for any $x\in\mathcal{C}\cap\overline{B_{R_0}(0)}$, thus we have proved \eqref{prop:3.2} and Proposition \ref{prop:3.1}.
\end{proof}

Based on Proposition \ref{prop:3.1}, we are to prove sharp asymptotic estimates on $u$ and $\nabla u$ at infinity in cones.
\begin{prop}\label{Prop:3.2}
    Let $u$ be a solution of \eqref{eq:3.1}, then there exist $R_0>0$ and $\beta>N$ such that
    \begin{equation}\label{eq:3.15}
        u(x)+\beta\log \widehat{H}_0(x)\in L^{\infty}(\overline{\mathcal{C}}\setminus B^{\widehat{H}_{0}}_{R_0}(0)),
    \end{equation}
    and
    \begin{equation}\label{eq:3.16}
        \lim_{x\in\mathcal{C}, \, |x|\rightarrow+\infty}|x|\left\lvert \nabla (u(x)+\beta\log\widehat{H}_0(x) )\right\rvert=0.
    \end{equation}
\end{prop}
\begin{proof}
Let $R_0>1$ be fixed, and $T^{\widehat{H}_{0}}_{R,R_0}:=B^{\widehat{H}_{0}}_{R}(0)\setminus\overline{B^{\widehat{H}_{0}}_{R_0}(0)}$. For $x\in\mathcal{C}\setminus B_{\frac{R_0}{R}}^{\widehat{H}_0}(0)$, we define $u_R(x):=\frac{u(Rx)-\mu_{0}}{\log R}$ and $f_R(x):=e^{u(Rx)}$, where $\mu_{0}:=\inf\limits_{\partial B^{\widehat{H}_0}_{R_0}(0)\cap\mathcal{C}}u$. From \eqref{prop:3.2}, we infer that
\begin{equation}\label{eq-1}
  U_0:=\sup\limits_{\mathcal{C}}u<+\infty \qquad \text{and} \qquad 0<e^{u(x)}\leq\frac{C}{|x|^N} \qquad \text{in} \quad\mathcal{C}.
\end{equation}
Let $\widehat{u}:=U_0-u$. By \eqref{prop:3.2}, we have $\widehat{u}(x)\geq N\log|x|+U_{0}-C\rightarrow+\infty$, as $|x|\rightarrow+\infty$ with $x\in\mathcal{C}$. Moreover, it follows from equation \eqref{eq:3.1} that, $\widehat{u}$ satisfies the equation
	\begin{equation}\label{eq:3.17}
		\left\{
		\begin{aligned}
			&\Delta ^{\widehat{H}}_{N}\widehat{u}=e^u \quad\,\,&{\rm{in}}\,\,\mathcal{C},\\
			&\widehat{a}(\nabla\widehat{u})\cdot\nu=0  \quad\,\,&{\rm{on}}\,\,\partial\mathcal{C},
		\end{aligned}
		\right.
	\end{equation}
	where $\widehat{a}(\nabla\widehat{u}):=\widehat{H}^{N-1}(\nabla\widehat{u})\nabla \widehat{H}(\nabla\widehat{u})$. Note that \eqref{norm0} yields
    \begin{equation*}
            \frac{1}{C_{H}}|x|\leq \widehat{H}_0(x)\leq \frac{1}{c_{H}}|x|, \qquad \forall \,\, x\in\mathbb{R}^{N}.
    \end{equation*}
    Thus by letting $x=\nabla \widehat{H}(\xi)$, we infer from $\widehat{H}_0(\nabla \widehat{H}(\xi))=1$ (see e.g. Lemma 2.3 in \cite{CFV}) that
    \begin{equation}\label{lub}
        c_{H}\leq \left\lvert \nabla \widehat{H}(\xi)\right\rvert \leq C_{H}, \qquad \forall \,\, \xi\in\mathbb{R}^{N},
    \end{equation}
    where the constants $C_{H}:=\max\limits_{\xi\in \mathbb{S}^{N-1}}H(\xi)\geq c_{H}:=\min\limits_{\xi\in \mathbb{S}^{N-1}}H(\xi)>0$. Consequently, $\widehat{a}:\,\mathbb{R}^{N}\rightarrow\mathbb{R}^{N}$ satisfies the condition \eqref{eq-3}. From the logarithmic asymptotic behavior in convex cone $\mathcal{C}$ at infinity in Theorem \ref{le:1+}, we deduce that, there exists constant $L\geq N$ such that
	\begin{equation}\label{asym-infty}
		N\log|x|+U_{0}-C\leq \widehat{u}\leq L\log|x|,
	\end{equation}
and hence
	\begin{equation}\label{asym-infty-u}
		U_{0}-L\log |x|\leq u\leq C-N\log|x|
	\end{equation}
for any $x\in\mathcal{C}$ with $|x|$ large enough. Therefore, we infer from \eqref{asym-infty-u} and the definition of $u_R(x)$ that
\begin{equation}\label{prop:u_R}
	|u_R(x)|\leq \frac{L\log|Rx|+C_1}{\log R}
\end{equation}
for some constant $C_1>0$, hence $u_R$ is bounded in $L^{\infty}_{loc}(\overline{\mathcal{C}}\setminus\{0\})$, uniformly in $R$. Once again, from \eqref{eq:3.1}, we know that $u_R$ satisfies the equation
	\begin{equation}\label{eq-u_R}
	\left\{
	\begin{aligned}
		&-\Delta ^{H}_{N}u_R=\frac{R^N}{(\log R)^{N-1}}f_R \quad\,\,&{\rm{in}}\,\,\mathcal{C}\setminus B_{\frac{R_0}{R}}^{\widehat{H}_0}(0),\\
		&a(\nabla u_R)\cdot\nu=0  \quad\,\,&{\rm{on}}\,\,\partial\mathcal{C}\setminus \overline{ B_{\frac{R_0}{R}}^{\widehat{H}_0}(0)}.
	\end{aligned}
	\right.
\end{equation}
Note that \eqref{eq-1} yields that $|f_R(x)|=|e^{u(Rx)}|\leq\frac{C_2}{|Rx|^N}$ for some constant $C_2>0$, which implies that
$$\left|\frac{R^N}{(\log R)^{N-1}}f_R\right|\leq\frac{C_2}{(\log R)^{N-1}|x|^N},$$
hence $\frac{R^N}{(\log R)^{N-1}}f_R$ is bounded in $L^{\infty}_{loc}(\overline{\mathcal{C}}\setminus\{0\})$, uniformly in $R$. By \eqref{eq-u_R}, the elliptic estimates in \cite{D,S2,T} and the boundary regularity in \cite{L}, we obtain that $u_R$ is uniformly bounded in $C^{1,\theta}_{loc}(\overline{\mathcal{C}}\setminus\{0\})$. By the Ascoli--Arzel\'a's Theorem and a diagonal argument, we can
find a sequence $R_j\rightarrow+\infty$ such that $u_{R_j}\rightarrow u_{\infty}$ in $C^{1,\theta}_{loc}(\overline{\mathcal{C}}\setminus\{0\})$ as $j\rightarrow+\infty$, where $u_{\infty}$ satisfies
 	\begin{equation}\label{eq-u_infty}
 	\left\{
 	\begin{aligned}
 		&\Delta ^{H}_{N}u_{\infty}=0 \quad\,\,&{\rm{in}}\,\,\mathcal{C},\\
 		&a(\nabla u_{\infty})\cdot\nu=0  \quad\,\,&{\rm{on}}\,\,\partial\mathcal{C}.
 	\end{aligned}
 	\right.
 \end{equation}
In fact, we have proved that, for any sequence $R_j\rightarrow+\infty$, there exists a subsequence (still denoted by $R_{j}$) such that $u_{R_j}\rightarrow u_{\infty}$ in $C^{1,\theta}_{loc}(\overline{\mathcal{C}}\setminus\{0\})$ as $j\rightarrow+\infty$, thus $u_{R}\rightarrow u_{\infty}$ in $C^{1,\theta}_{loc}(\overline{\mathcal{C}}\setminus\{0\})$ as $R\rightarrow+\infty$. Moreover, it follows from \eqref{prop:u_R} that $|u_{\infty}|\leq L$ for any $x\in\mathcal{C}$. Then by Lemma \ref{le:2.8} with $\gamma=0$ and \eqref{asym-infty-u}, we obtain that $u_{\infty}\equiv-\beta$ for some positive constant $N\leq\beta\leq L$.

\smallskip

Next, we will prove that
\begin{equation}\label{L-infty-u}
	u(x)+\beta\log \widehat{H}_0(x)\in L^{\infty}(\overline{\mathcal{C}}\setminus B_{R_0}^{\widehat{H}_0}(0)).
\end{equation}
To this end, for any $\varepsilon>0$ small enough, let
$$\underline{u}_{\varepsilon}(x):=(-\beta-\varepsilon)\log\left[\frac{\widehat{H}_0(x)}{R_0}\right]+\inf\limits_{\partial B_{R_0}^{\widehat{H}_0}(0)\cap\mathcal{C}}(u-\mu_{0}).$$
Since $u_{\infty}\equiv-\beta$, then there exists $R(\varepsilon)>R_{0}$ such that
\begin{equation}\label{a-0}
  \underline{u}_{\varepsilon}(x)\leq u(x)-\mu_{0}\leq(-\beta+\varepsilon)\log\widehat{H}_0(x),\qquad\forall\,\,x\,\in\mathcal{C}\setminus B^{\widehat{H}_0}_{R(\varepsilon)}(0).
\end{equation}
and hence, for any $R\geq R(\varepsilon)$, there holds
$$u-\mu_{0}\geq \underline{u}_{\varepsilon} \quad\,\,\, \,\,\, \text{on}\,\, \mathcal{C}\cap\left(\partial B_{R_0}^{\widehat{H}_0}(0)\cup\partial B_{R}^{\widehat{H}_0}(0)\right).$$
By \eqref{eq:2.2}, one has $\left\langle a\left(\nabla\big(-\log\widehat{H}_0(x)\big)\right), \nu\right\rangle=\frac{1}{\widehat{H}_0(x)^N}\left\langle x,\nu\right\rangle=0$ a.e. on $\partial\mathcal{C}$. Thus by applying Lemma \ref{le:2.5} in $\mathcal{C}\cap T^{\widehat{H}_{0}}_{R,R_0}:=\mathcal{C}\cap\left(B^{\widehat{H}_{0}}_{R}(0)\setminus\overline{B^{\widehat{H}_{0}}_{R_0}(0)}\right)$, then letting $R\rightarrow+\infty$ and $\varepsilon\rightarrow0$, we deduce that, for any $x\in\mathcal{C}\setminus B_{R_0}^{\widehat{H}_0}(0)$, there holds
\begin{equation}\label{a-1}
  u(x)+\beta\log\widehat{H}_0(x)\geq \beta\log R_0+\inf\limits_{\partial B_{R_0}^{\widehat{H}_0}(0)\cap\mathcal{C}}(u-\mu_{0})+\mu_{0},
\end{equation}
which combining with $\int_{\mathcal{C}}e^{u}\mathrm{d}x<+\infty$ implies immediately that $\beta>N$. By taking $0<\varepsilon=\frac{\beta-N}{2}<\beta-N$ in \eqref{a-0}, we get that, there exists $C_0>0$ such that
$$e^{u(x)}\leq C_0\left[\widehat{H}_0(x)\right]^{-\frac{N+\beta}{2}},\qquad\forall\,\,x\in\mathcal{C}\setminus B^{\widehat{H}_0}_{R_0}.$$
For $\varepsilon\in(0,\beta-N)$ small enough, let
$$\overline{u}_{\varepsilon}(x):=C_0^{\frac{1}{N-1}}\phi _{a_{\varepsilon}, b}\left(\widehat{H}_0(x)\right),$$
where $\phi_{a_{\varepsilon},b}(r):=-\left(\frac{\beta-N}{2}\right)^{\frac{1}{1-N}}\int_{R_1}^{r}\frac{\left(a_{\varepsilon}-s^{\frac{N-\beta}{2}}\right)^{\frac{1}{N-1}}}{s}\mathrm{d}s+b$ solves the ODE
$r^{1-N}\left(r^{N-1}(-\phi ^{\prime})^{N-1}\right)^{\prime}=r^{-\frac{N+\beta}{2}}$ ($r\geq R_1$), $a_{\varepsilon}:=\frac{(\beta-\varepsilon)^{N-1}(\beta-N)}{2C_0}$, $R_1$ is fixed suitably large such that $a_{\varepsilon}>\frac{N^{N-1}(\beta-N)}{2C_0}\geq2 R_1^{\frac{N-\beta}{2}}$, $R_1\geq R_{0}$ and $N\log R_{1}\geq\mu_0$, $b$ is chosen suitably large such that
$C_0^{\frac{1}{N-1}}b\geq\max\Big\{0,\max\limits_{\partial B^{\widehat{H}_0}_{R_1}\cap\mathcal{C}}u\Big\}$ (see also the proof of \cite[Proposition 3.1]{CL2}). One can verify that
$$-\Delta^{H}_{N}\overline{u}_{\varepsilon}=C_0\left[\widehat{H}_0(x)\right]^{-\frac{N+\beta}{2}}\geq e^u=-\Delta^{H}_{N}u\qquad\mathrm{in}\,\,T^{\widehat{H}_{0}}_{R,R_1}\cap\mathcal{C},$$
and
\begin{equation}\label{a-2}
  (-\beta+\varepsilon)\log\left[\frac{\widehat{H}_0(x)}{R_1}\right]+C_0^{\frac{1}{N-1}}b\leq \overline{u}_{\varepsilon}(x)\leq (-\beta+\varepsilon)\log\left[\frac{\widehat{H}_0(x)}{R_1}\right]+\widetilde{C}_{1}+C_{0}^{\frac{1}{N-1}}b
\end{equation}
for all $x\in\mathcal{C}\setminus B^{\widehat{H}_0}_{R_1}$ and $\varepsilon\in(0,\beta-N)$ small enough, where $\widetilde{C}_{1}:=\frac{2}{(N-1)(\beta-N)}\left(\frac{2C_{0}R_{1}^{\frac{N-\beta}{2}}}{\beta-N}\right)^{\frac{1}{N-1}}$. It follows from \eqref{a-0} and \eqref{a-2} that, for any $R\geq R(\varepsilon)>R_{1}$, there holds
$$ u\leq \overline{u}_{\varepsilon}\quad\,\,\, \,\,\, \text{on}\,\, \mathcal{C}\cap\left(\partial B_{R_1}^{\widehat{H}_0}(0)\cup\partial B_{R}^{\widehat{H}_0}(0)\right).$$
Moreover, by Lemma \ref{8,le:2.2}, one has $a(\nabla \overline{u}_{\varepsilon})\cdot\nu=0$ a.e. on $\partial\mathcal{C}\setminus B_{R_1}^{\widehat{H}_0}(0)$. Thus by applying Lemma \ref{le:2.5} in $\mathcal{C}\cap T^{\widehat{H}_{0}}_{R,R_1}:=\mathcal{C}\cap\left(B^{\widehat{H}_{0}}_{R}(0)\setminus\overline{B^{\widehat{H}_{0}}_{R_1}(0)}\right)$ and using \eqref{a-2}, then letting $R\rightarrow+\infty$ and $\varepsilon\rightarrow0$, we deduce that, for any $x\in\mathcal{C}\setminus B_{R_1}^{\widehat{H}_0}(0)$, there holds
$$ u(x)+\beta\log\widehat{H}_0(x)\leq\beta\log R_1+\frac{2}{(N-1)(\beta-N)}\left(\frac{2C_{0}R_{1}^{\frac{N-\beta}{2}}}{\beta-N}\right)^{\frac{1}{N-1}}+C_{0}^{\frac{1}{N-1}}b,$$
which combining with \eqref{a-1} implies \eqref{L-infty-u} immediately and hence the validity of \eqref{eq:3.15}.

\smallskip

Now, let $G(x):=u(x)+\beta\log \widehat{H}_0(x)$, then \eqref{L-infty-u} implies that $G(x)\in L^{\infty}(\overline{\mathcal{C}}\setminus B^{\widehat{H}_0}_{R_0}(0))$, and
 \begin{equation}\label{prop:3.25}
    u(x)=-\beta\log \widehat{H}_0(x)+G(x)\ \quad\,\,\, \text{for}\,\,\, x\in\mathcal{C}\setminus B^{\widehat{H}_0}_{R_0}(0),
 \end{equation}
where $N<\beta\leq L$.

\medskip

For any sequence $R_n>0$ with $R_n\rightarrow+\infty$ as $n\rightarrow+\infty$, let
 \begin{equation*}
    \widehat{u}_{R_n}(x):=u(R_nx)+\beta\log R_n.
 \end{equation*}
 By \eqref{prop:3.25}, we know that
 \begin{equation}\label{eq:3.26}
    \widehat{u}_{R_n}(x)=-\beta\log\widehat{H}_0(x)+G_{R_n}(x),
 \end{equation}
 where $G_{R_n}(x)=G(R_nx)$. In view of \eqref{eq:3.1}, we obtain that
 \begin{equation*}
    \left\{
        \begin{aligned}
        & -\Delta ^{H}_{N} \widehat{u}_{R_n}=R_{n}^{N-\beta}\frac{e^{G_{R_n}}}{\widehat{H}_0^{\beta}} \qquad \ &{\rm{in}}\,\, \mathcal{C}\setminus B^{\widehat{H}_0}_{\frac{R_0}{R_n}}(0),\\
        &a(\nabla  \widehat{u}_{R_n})\cdot \nu =0\qquad &{\rm{on}}\,\, \partial\mathcal{C}\setminus B^{\widehat{H}_0}_{\frac{R_0}{R_n}}(0).
        \end{aligned}
        \right.
 \end{equation*}
 Since $ \widehat{u}_{R_n}$ and $R_{n}^{N-\beta}\frac{e^{G_{R_n}}}{\widehat{H}_0^{\beta}}$ are bounded in $L^{\infty}_{loc}(\overline{\mathcal{C}}\setminus \{0\})$, uniformly in $R_n$,
 by the interior regularity estimates in \cite{D,S2,T} and the boundary regularity estimates in \cite{L}, we obtain that $ \widehat{u}_{R_n}$ is uniformly bounded in $C^{1,\theta}_{loc}(\overline{\mathcal{C}}\setminus \{0\})$.
 Then by the Ascoli--Arzel\'a's Theorem and a diagonal argument, we can find a subsequence $R_{n_i}\rightarrow+\infty$ so that $ \widehat{u}_{R_{n_i}}\rightarrow  \widehat{u}_{\infty}$ in $C^{1,\theta}_{loc}(\overline{\mathcal{C}}\setminus \{0\})$ as $i\rightarrow+\infty$. By $\beta>N$, we deduce that $ \widehat{u}_{\infty}$ satisfies
 \begin{equation}\label{eq:3.27}
    \left\{
        \begin{aligned}
        & -\Delta ^{H}_{N} \widehat{u}_{\infty}=0\qquad \ &{\rm{in}}\,\, \mathcal{C},\\
        &a(\nabla  \widehat{u}_{\infty})\cdot \nu =0\qquad &{\rm{on}}\,\, \partial\mathcal{C}.
        \end{aligned}
        \right.
 \end{equation}
 In view of \eqref{eq:3.26}, we have
 \begin{equation}\label{eq:3.28}
    \widehat{u}_{\infty}(x)=-\beta\log \widehat{H}_0(x)+G_{\infty}(x),
 \end{equation}
 where $G_{\infty}$ is the limit of $G_{R_{n_i}}$ in $C^{1}_{loc}(\overline{\mathcal{C}}\setminus \{0\})$. Moreover, the uniform boundedness of $G_{R_n}$ implies $G_{\infty}\in L^{\infty}(\overline{\mathcal{C}}\setminus \{0\})$. By \eqref{eq:3.27}, \eqref{eq:3.28} and $G_{\infty}\in L^{\infty}(\overline{\mathcal{C}}\setminus \{0\})$, we can apply Lemma \ref{le:2.8} to $\widehat{u}_{\infty}$ and derive that $G_{\infty}$ is a constant function.
 By \eqref{prop:3.25} and $G_{R_{n_i}}\rightarrow G_{\infty}$ in $C^{1}_{loc}(\overline{\mathcal{C}}\setminus \{0\})$, we know that
 \begin{equation*}
    \sup_{|x|=R_{n_i},\,x\in\mathcal{C}}|x|\left\lvert \nabla (u(x)+\beta\log\widehat{H}_0(x) )\right\rvert
    =\sup_{|y|=1,\,y\in\mathcal{C}}\left\lvert \nabla G_{R_{n_i}}(y)\right\rvert\rightarrow \sup_{|y|=1,\,y\in\mathcal{C}}\left\lvert \nabla G_{\infty}(y)\right\rvert=0,
 \end{equation*}
 which implies the validity of \eqref{eq:3.16}, thus we have finished our proof.
\end{proof}

\section{Second-order regularity for weak solutions}
In this section, we will prove the second-order regularity for weak solutions via an approximation method.
\begin{prop}\label{Proposition:4.1}
   Let $\mathcal{C}$ be a convex cone and $u$ be a weak solution of \eqref{eq:1.1}, where the gauge $H$ is the same as in Theorem \ref{Th:1.1} and the function $a$ is given by \eqref{a}. Then $a(\nabla u)\in W^{1,2}_{loc}(\overline{\mathcal{C}})$.
\end{prop}
\begin{proof}
    The estimate $a(\nabla u)\in W^{1,2}_{loc}(\mathcal{C})$ can be obtained from \cite[Theorem 4.2]{AKM}. Our task is to show further that $a(\nabla u)\in W^{1,2}_{loc}(\overline{\mathcal{C}})$. The main difficulties are the degeneracy of the equation and the non-smoothness of $\partial\mathcal{C}$. We will argue by an approximation method.

    Consider a sequence of convex cones $\{\mathcal{C}_k\}$ such that $\mathcal{C}_k\subset\mathcal{C}$, $\partial\mathcal{C}_k\setminus\{0\}$ is smooth and $\mathcal{C}_k$ approximate $\mathcal{C}$. Fix a point $\overline{x}\in \bigcap \limits_{k}\mathcal{C}_k$, and for $k$ fixed we let $u_k$ be the solution of
    \begin{equation}\label{eq:4.1}
        \left\{
            \begin{aligned}
            & -\Delta ^{H}_{N}u_{k}=e^u\quad \, &{\rm{in}}\ \mathcal{C}_k,\\
            &u_k(\overline{x})=u(\overline{x}),\\
            &a(\nabla u_{k})\cdot \nu =0\quad\, &{\rm{on}}\ \partial\mathcal{C}_k.
            \end{aligned}
            \right.
    \end{equation}
    Since Proposition \ref{prop:3.1} implies that $u$ is locally bounded, by the elliptic estimates in \cite{L,S,T}, we obtain
    that $\{u_k\}\subset C^{1,\theta}_{loc}(\overline{\mathcal{C}_k}\setminus\{0\})\cap C^{0,\theta}_{loc}(\overline{\mathcal{C}_k})$ and
    for every compact subset $A\subset \mathcal{C}$, $\|u_{k}\|_{C^{1,\theta}(A)}$ is uniformly bounded with respect to $k$. Then
    by the Ascoli--Arzel$\acute{\text{a}} $'s Theorem and a diagonal process, we have
    \begin{equation}\label{prop:4.2}
      u_k\rightarrow v\quad \text{and}\quad \nabla u_k\rightarrow \nabla v\,\,\, \text{pointwise\ in}\,\,\, A,\qquad \text{as}\ k\rightarrow+\infty.
    \end{equation}

\smallskip

In view of \eqref{eq:4.1}, we deduce that $v$ satisfies
    \begin{equation}\label{eq:4.3}
        \left\{
            \begin{aligned}
            & -\Delta ^{H}_{N}v=e^u\quad \, &{\rm{in}}\,\, \mathcal{C},\\
            &v(\overline{x})=u(\overline{x}),\\
            &a(\nabla v)\cdot \nu =0\quad\, &{\rm{on}}\,\, \partial\mathcal{C}.
            \end{aligned}
            \right.
    \end{equation}
    Note that the solution of \eqref{eq:4.3} is unique. Actually, assume that $u_1$, $v_1$ are two solutions of equation \eqref{eq:4.3}, then we have
    \begin{equation*}
        \left\{
            \begin{aligned}
            & -\Delta ^{H}_{N}u_1+\Delta ^{H}_{N}v_1=0\quad \, &{\rm{in}}\,\, \mathcal{C},\\
            &u_1(\overline{x})=v_1(\overline{x})=u(\overline{x}), \\
            &a(\nabla u_1)\cdot \nu =a(\nabla v_1)\cdot \nu=0\quad\, &{\rm{on}}\,\, \partial\mathcal{C}.
            \end{aligned}
            \right.
    \end{equation*}
Choosing $u_1-v_1$ as a test function, we get
$$\int_{\mathcal{C}}\langle a(\nabla u_1)-a(\nabla v_1), \nabla u_1-\nabla v_1\rangle {\rm{d}}x-\int_{\partial\mathcal{C}} (a(\nabla u_1)-a(\nabla v_1))\cdot \nu(u_1-v_1){\rm{d}}\mathcal{H}^{N-1}=0.$$
Note that $(a(\nabla u_1)-a(\nabla v_1))\cdot \nu=0$ on $\partial\mathcal{C}$, and let $\xi_1=\nabla u_1$, $\xi_2=\nabla v_1$ in \eqref{aeq:2.4}, we have
\begin{equation*}
        c_1\int_{\mathcal{C}}(|\nabla u_1|+|\nabla v_1|)^{N-2}|\nabla(u_1-v_1)|^2{\rm {d}}x\leq\int_{\mathcal{C}}\langle a(\nabla u_1)-a(\nabla v_1), \nabla u_1-\nabla v_1\rangle {\rm{d}}x=0,
\end{equation*}
and hence $\nabla(u_1-v_1)=0$. Since convex cone $\mathcal{C}$ is connected, then $u_1=v_1+C_1$ for some constant $C_1$. Moreover, it follows from $u_1(\overline{x})=v_1(\overline{x})=u(\overline{x})$ that $C_1=0$, hence the solution of \eqref{eq:4.3} is unique. Since both $u$ and $v$ are solutions of \eqref{eq:4.3}, thus we must have $v\equiv u$. Then from \eqref{prop:4.2}, we infer that
\begin{equation}\label{4+}
    u_k\rightarrow u\quad \text{and}\quad \nabla u_k\rightarrow \nabla u\,\,\, \text{pointwise\ in}\,\,\, A,\qquad \text{as}\ k\rightarrow+\infty.
\end{equation}

\smallskip

Let $\{\phi _l\}$ with $l\in\mathbb{N}$ be a family of radially symmetric smooth mollifiers and define
\begin{equation*}
    a^l(z):=(a\ast \phi_l)(z)\qquad \text{for}\,\, z\in\mathbb{R}^N,
\end{equation*}
where $a\ast\phi_l$ stands for the convolution. By the properties of convolution and continuity of $a(\cdot)$, we have
 $a^l\rightarrow a$ uniformly on compact subset of $\mathbb{R}^N$, as $l\rightarrow+\infty$. By \cite[Lemma 2.4]{FF} and its proof therein, we know that
 $a^l$ satisfies
 \begin{equation}\label{aeq:4.3}
    \left\langle \nabla a^l(z)\xi,\xi\right\rangle\geq\frac{1}{\lambda}(|z|^2+s_l^2)^{\frac{N-2}{2}}|\xi|^2
    \qquad \text{and}\qquad |\nabla a^l(z)|\leq \lambda (|z|^2+s_l^2)^{\frac{N-2}{2}}
 \end{equation}
for every $z$, $\xi\in\mathbb{R}^N$, with $s_l\neq0$ and $s_l\rightarrow0$ as $l\rightarrow+\infty$, and $\lambda>0$ be given by \eqref{ineq:2.3} in Lemma \ref{8,le:2.3}.

Let $u_{l,k}\in W^{1,N}_{loc}(\mathcal{C}_k)$ be a weak solution of
\begin{equation}\label{eq:4.4}
    \left\{
            \begin{aligned}
            & -\Delta ^{H,l}_{N}u_{l,k}=-\operatorname{div}\left(a^l(\nabla u_{l,k})\right)=e^u\quad \, &{\rm{in}}\,\, \mathcal{C}_k,\\
            &u_{l,k}(\overline{x})=u(\overline{x}),\\
            &a^l(\nabla u_{l,k})\cdot \nu =0\quad\, &{\rm{on}}\,\, \partial\mathcal{C}_k.
            \end{aligned}
            \right.
\end{equation}
By the locally boundedness of $u$ and the elliptic estimates in \cite{L,S,T}, we deduce that $\{u_{l,k}\}$ are bounded in
 $C^{1,\theta}_{loc}(\overline{\mathcal{C}_k}\setminus\{0\})\cap C^{0,\theta}_{loc}(\overline{\mathcal{C}_k})$
  uniformly in $l$, as $l\rightarrow+\infty$. Then by the Ascoli--Arzel$\acute{\text{a}}$'s Theorem and a diagonal process, we obtain that $u_{l,k}$
   converges in $C^1_{loc}(\overline{\mathcal{C}_k}\setminus\{0\})$ to the unique solution $\overline{u}_k$ of
   \begin{equation}\label{eq:4.5}
    \left\{
        \begin{aligned}
        & -\Delta ^{H}_{N}\overline{u}_k=e^u\quad \, &{\rm{in}}\,\, \mathcal{C}_k,\\
        &\overline{u}_k(\overline{x})=u(\overline{x}),\\
        &a(\nabla \overline{u}_k)\cdot \nu =0\quad\, &{\rm{on}}\,\, \partial\mathcal{C}_k.
        \end{aligned}
        \right.
   \end{equation}
Note that $u_k$ is also a solution of \eqref{eq:4.5}, it follows from uniqueness that $\overline{u}_k=u_k$, and hence $u_{l,k}\rightarrow u_k$ in $C^1_{loc}(\overline{\mathcal{C}_k}\setminus\{0\})$ as $l\rightarrow+\infty$.

\smallskip

Given $R>1$ large, we define
 \begin{equation*}
    \mathcal{C}_{k,R}:=\mathcal{C}_k\cap B_R(0),\quad\, \Gamma_{k,0}^R:=\mathcal{C}_k\cap \partial B_R(0),\quad\, \Gamma_{k,1}^R:=\partial\mathcal{C}_k\cap  B_R(0).
 \end{equation*}
 For $\delta>0$ small enough, such that $2\delta<\frac{1}{R}$, define the set
 \begin{equation*}
    \mathcal{C}_{k,R}^{\delta}:=\{x\in\mathcal{C}_{k,R}\mid\,\text{dist}(x,\partial\mathcal{C}_{k,R})>\delta\}.
 \end{equation*}
 \par Let $\varphi \in C^{1}_c(B_R(0)\setminus B_{\frac{1}{R}}(0))$.  Since $\mathcal{C}_{k,R}\cap \text{supp}(\varphi)$
 is piecewise smooth, for $\delta>0$ small we can obtain
 that $(\mathcal{C}_{k,R}^{\delta}\setminus \mathcal{C}_{k,R}^{2\delta})\cap\text{supp}(\varphi)$ is piecewise smooth. In particular,
 every $x\in(\mathcal{C}_{k,R}^{\delta}\setminus \mathcal{C}_{k,R}^{2\delta})\cap\text{supp}(\varphi)$ can be written as
 \begin{equation*}
    x=y-|x-y|\nu(y),
 \end{equation*}
 where $y=y(x)\in\partial\mathcal{C}_{k,R}^{\delta}$ is the unique projection of $x$ on $\partial\mathcal{C}_{k,R}^{\delta}$ and $\nu(y)$
 is the unit outer normal vector to $\partial\mathcal{C}_{k,R}^{\delta}$ at $y$. Moreover, by the method in \cite[Formula (14.98)]{GT}, we can
  define a $C^1$ function $g$, such that $x=g(y,d)$ holds for every
   $x\in(\mathcal{C}_{k,R}^{\delta}\setminus \mathcal{C}_{k,R}^{2\delta})\cap\text{supp}(\varphi)$, where $y\in\partial\mathcal{C}_{k,R}^{\delta}$ is the projection of $x$ on $\partial\mathcal{C}_{k,R}^{\delta}$ and $d=|x-y|$. Since $u_{l,k}$
     solves the non-degenerate equation, we have $u_{l,k}\in C^1\cap W^{2,2}_{loc}(\overline{\mathcal{C}_k})$ and
     $a^l(\nabla u_{l,k})\in W^{1,2}_{loc}(\overline{\mathcal{C}_k})$. Moreover, since
       $\mathcal{C}_{k,R}^{\delta}\setminus \mathcal{C}_{k,R}^{2\delta}$ is piecewise smooth, then $u_{l,k}\in C^2(\mathcal{C}_{k,R}^{\delta}\setminus \mathcal{C}_{k,R}^{2\delta})$.

\smallskip

Let $\zeta _{\delta}: \mathcal{C}_{k,R}\rightarrow[0,1]$ be a piecewise smooth function such that
\begin{equation*}
    \zeta _{\delta}=1\ \text{in}\ \mathcal{C}_{k,R}^{2\delta},\quad
    \zeta _{\delta}=0\ \text{in}\ \mathcal{C}_{k,R}\setminus\mathcal{C}_{k,R}^{\delta},\quad \text{and}\quad
    \nabla\zeta _{\delta}=-\frac{1}{\delta}\nu(y(x))\,\,\, \text{in}\,\, \mathcal{C}_{k,R}^{\delta}\setminus\mathcal{C}_{k,R}^{2\delta}.
\end{equation*}
Choosing $\psi \in C^{1}_c(B_R(0)\setminus B_{\frac{1}{R}}(0))$ as the test function in \eqref{eq:4.4}, and integrating by
parts, we get
\begin{equation}\label{eq:4.6}
    \int_{\mathcal{C}_{k,R}}a^l(\nabla u_{l,k})\cdot\nabla \psi \mathrm{d}x-\int_{\partial\mathcal{C}_{k,R}}\psi a^l(\nabla u_{l,k})\cdot\nu \mathrm{d}\mathcal{H}^{N-1}
    =\int_{\mathcal{C}_{k,R}}e^u\psi \mathrm{d}x.
\end{equation}
Since
\begin{equation*}
    \int_{\partial\mathcal{C}_{k,R}}\psi a^l(\nabla u_{l,k})\cdot\nu \mathrm{d}\mathcal{H}^{N-1}=\int_{\Gamma_{k,0}^R}\psi a^l(\nabla u_{l,k})\cdot\nu \mathrm{d}\mathcal{H}^{N-1}
    +\int_{\Gamma_{k,1}^R}\psi a^l(\nabla u_{l,k})\cdot\nu \mathrm{d}\mathcal{H}^{N-1},
\end{equation*}
in view of $\psi \in C^{\infty}_c(B_R(0)\setminus B_{\frac{1}{R}}(0))$ and the boundary condition in \eqref{eq:4.4}, we obtain
 that the boundary integral term in \eqref{eq:4.6} vanishes, thus \eqref{eq:4.6} becomes
 \begin{equation}\label{eq:4.7}
    \int_{\mathcal{C}_{k,R}}a^l(\nabla u_{l,k})\cdot\nabla \psi \mathrm{d}x= \int_{\mathcal{C}_{k,R}}e^u\psi \mathrm{d}x.
 \end{equation}

Next, we choose $\psi=\partial_m(\varphi\zeta_{\delta})$ in \eqref{eq:4.7} with $m\in\{1,\cdots,N\}$, then obtain that
\begin{equation*}
    \sum\limits_{i=1}^{N}\left(\int_{\mathcal{C}_{k,R}}\partial_m (a^l_i(\nabla u_{l,k}))\zeta_{\delta}\partial_i\varphi \mathrm{d}x+\int_{\mathcal{C}_{k,R}}\partial_m (a^l_i(\nabla u_{l,k}))\varphi\partial_i\zeta_{\delta} \mathrm{d}x\right)
    = \int_{\mathcal{C}_{k,R}}\partial_m e^u\zeta_{\delta}\varphi \mathrm{d}x,
\end{equation*}
where the notation $a^l=(a^l_1,\cdots,a^l_N)$ denotes the components of the vector field $a^l$. Recalling the definition of $\zeta_{\delta}$, we have
\begin{equation*}
    \lim\limits_{\delta\rightarrow0}\int_{\mathcal{C}_{k,R}}\partial_m (a^l_i(\nabla u_{l,k}))\zeta_{\delta}\partial_i\varphi \mathrm{d}x=
    \int_{\mathcal{C}_{k,R}}\partial_m (a^l_i(\nabla u_{l,k}))\partial_i\varphi \mathrm{d}x.
\end{equation*}

Let $f(x)=\partial_m (a^l_i(\nabla u_{l,k}))\varphi(x) $, by $\nabla\zeta _{\delta}=-\frac{1}{\delta}\nu(y)$ and the co-area formula, we have
 \begin{equation*}
    \begin{aligned}
        \int_{\mathcal{C}_{k,R}^{\delta}\setminus \mathcal{C}_{k,R}^{2\delta}}f\partial_i\zeta_{\delta}\mathrm{d}x
            &=-\frac{1}{\delta} \int_{\mathcal{C}_{k,R}^{\delta}\setminus \mathcal{C}_{k,R}^{2\delta}}\nu_i(y(x))f\mathrm{d}x\\
            &=-\frac{1}{\delta}\int_{\delta}^{2\delta}\mathrm{d}t\int_{\partial\mathcal{C}_{k,R}^{t}}\nu_i(y(x))f(y-t\nu(y))|\det(Dg)|\mathrm{d}\mathcal{H}^{N-1}\\
            &=-\int_{1}^{2}\mathrm{d}s\int_{\partial\mathcal{C}_{k,R}^{s\delta}}\nu_i(y(x))f(y-s\delta\nu(y))|\det(Dg)|\mathrm{d}\mathcal{H}^{N-1}.
    \end{aligned}
 \end{equation*}
Since $\nabla u_{l,k}\in C^1(\mathcal{C}_{k,R}^{\delta}\setminus\mathcal{C}_{k,R}^{2\delta})$, one has $f(x)=\partial_m (a^l_i(\nabla u_{l,k}))\varphi(x) \in C^{0}$, hence we can pass to the limit and deduce that
\begin{equation*}
    \lim\limits_{\delta\rightarrow0}\int_{\mathcal{C}_{k,R}}\partial_m (a^l_i(\nabla u_{l,k}))\varphi\partial_i\zeta_{\delta} \mathrm{d}x
    =-\int_{\partial\mathcal{C}_{k,R}}\partial_m (a^l_i(\nabla u_{l,k}))\varphi(x)\nu_i \mathrm{d}\mathcal{H}^{N-1}.
\end{equation*}
Thus we obtain
\begin{equation}\label{eq:4.8}
    \sum\limits_{i=1}^{N} \left(\int_{\mathcal{C}_{k,R}}\partial_m (a^l_i(\nabla u_{l,k}))\partial_i\varphi \mathrm{d}x
    -\int_{\partial\mathcal{C}_{k,R}}\partial_m (a^l_i(\nabla u_{l,k}))\varphi(x)\nu_i \mathrm{d}\mathcal{H}^{N-1}\right)
     =\int_{\mathcal{C}_{k,R}}\partial_m e^u\varphi \mathrm{d}x.
\end{equation}

Choose $\varphi=a^l_m(\nabla u_{l,k})\rho ^2$, $m\in\{1,2,\cdots,N\}$, where $\rho \in C_c^{\infty}(B_R(0)\setminus B_{\frac{1}{R}}(0))$. Next, using the fact that $\mathcal{C}_{k,R}$ is convex and the argument as in the proof of \cite[Proposition 2.8]{CFR} (see Formulae (2.45)-(2.50) therein), we will show that the summation $\sum\limits_{m=1}^{N}$ over the second term on left-hand side of \eqref{eq:4.8} is negative. First, one has
\begin{equation}\label{aeq:4.9}
    \partial_m (a^l_i(\nabla u_{l,k}))a^l_m(\nabla u_{l,k})\rho ^2\nu_i=a_m^l(\nabla u_{l,k})\partial_m (a^l_i(\nabla u_{l,k})\nu_i)\rho^2-a_m^l(\nabla u_{l,k})a_i^l(\nabla u_{l,k})\rho^2\partial_m\nu_i.
\end{equation}
Note that $\partial_m\nu_i(x)$ is the second fundamental form $\mathrm{II}_x$ of $\Gamma^{R}_{k,1}$ at $x$:
$$\sum\limits_{i,m=1}^{N}\partial_m\nu_i a^l_i(\nabla u_{l,k})a^l_m(\nabla u_{l,k})=\mathrm{II}_x(a^l(\nabla u_{l,k}),a^l(\nabla u_{l,k})).$$
Since $\mathcal{C}_{k,R}$ is convex, then $\mathrm{II}_x$ is non-negative definite, we deduce that
\begin{equation*}
    \sum\limits_{i,m=1}^{N}\partial_m\nu_i a^l_i(\nabla u_{l,k})a^l_m(\nabla u_{l,k})\rho^2=\mathrm{II}_x(a^l(\nabla u_{l,k}),a^l(\nabla u_{l,k}))\rho^2\geq0.
\end{equation*}
Then, by \eqref{aeq:4.9}, we have
\begin{equation*}
    \begin{aligned}
        \sum\limits_{i,m=1}^{N}\int_{\partial\mathcal{C}_{k,R}}\partial_m (a^l_i(\nabla u_{l,k}))a^l_m(\nabla u_{l,k})\rho ^2\nu_i\mathrm{d}\mathcal{H}^{N-1}
        &\leq \sum\limits_{i,m=1}^{N}\int_{\partial\mathcal{C}_{k,R}}a_m^l(\nabla u_{l,k})\partial_m (a^l_i(\nabla u_{l,k})\nu_i)\rho^2\mathrm{d}\mathcal{H}^{N-1}\\
        &=\int_{\partial\mathcal{C}_{k,R}}a^l(\nabla u_{l,k})\cdot\nabla(a^l(\nabla u_{l,k})\cdot\nu)\rho^2\mathrm{d}\mathcal{H}^{N-1}.
    \end{aligned}
\end{equation*}
Note that $a^l(\nabla u_{l,k})\cdot\nu=0$ on $\Gamma_{k,1}^R$, thus $a^l(\nabla u_{l,k})$ is the tangent vector field of $\partial\Gamma_{k,1}^R$, which implies that $a^l(\nabla u_{l,k})\cdot\nabla(a^l(\nabla u_{l,k})\cdot\nu)$ is the tangent derivative of $a^l(\nabla u_{l,k})\cdot\nu$, then it follows from $a^l(\nabla u_{l,k})\cdot\nu=0$ on $\Gamma_{k,1}^R$ that $a^l(\nabla u_{l,k})\cdot\nabla(a^l(\nabla u_{l,k})\cdot\nu)=0$ on $\Gamma_{k,1}^R$. Since $\partial\mathcal{C}_{k,R}=\Gamma_{k,0}^R\cup\Gamma_{k,1}^R$ and $\rho=0$ on $\Gamma_{k,0}^R$, one has
$$\int_{\partial\mathcal{C}_{k,R}}a^l(\nabla u_{l,k})\cdot\nabla(a^l(\nabla u_{l,k})\cdot\nu)\rho^2\mathrm{d}\mathcal{H}^{N-1}=0,$$
thus we obtain that
\begin{equation*}
    \sum\limits_{i,m=1}^{N}\int_{\partial\mathcal{C}_{k,R}}\partial_m (a^l_i(\nabla u_{l,k}))a^l_{m}(\nabla u_{l,k})\rho ^2\nu_i \mathrm{d}\mathcal{H}^{N-1}\leq0.
\end{equation*}
Hence by \eqref{eq:4.8}, we have
\begin{equation*}
    \sum\limits_{i,m=1}^{N}\int_{\mathcal{C}_{k,R}}\partial_m (a^l_i(\nabla u_{l,k}))\partial_i(a^l_m(\nabla u_{l,k})\rho ^2)\mathrm{d}x
    \leq N\int_{\mathcal{C}_{k,R}}|\nabla e^u| |a^l(\nabla u_{l,k})|\rho ^2 \mathrm{d}x .
\end{equation*}
Let $A:=\nabla a^l(z)|_{z=\nabla u_{l,k}}=\nabla a^l(\nabla u_{l,k})$, $B:=\nabla^2 u_{l,k}(x)$. Then $A$, $B$ are symmetric matrices. Let $\lambda_{min}$ be the
smallest eigenvalue of $A$, and $\lambda_{max}$ be the largest eigenvalue of $A$, then by \eqref{aeq:4.3}, we know that
$$\frac{1}{\lambda}(|\nabla u_{l,k}|^2+s_l^2)^{\frac{N-2}{2}}\leq\lambda_{min}\leq\lambda_{max}\leq\lambda(|\nabla u_{l,k}|^2+s_l^2)^{\frac{N-2}{2}}.$$
Note that
$$\nabla (a^l(\nabla u_{l,k}))=\nabla a^l(\nabla u_{l,k})\nabla^2u_{l,k}=AB,$$
and
$$\mathrm{trace} [(\nabla (a^l(\nabla u_{l,k})))^2]=\sum\limits_{i,m=1}^{N}\partial_i (a^l_m(\nabla u_{l,k}))\partial_m (a^l_i(\nabla u_{l,k})),$$
thus we can deduce from  $\frac{\lambda_{max}}{\lambda_{min}}\leq\lambda^{2}$ and Lemma \ref{le:2.9} that
\begin{equation*}
    \begin{aligned}
        N\lambda^4 \mathrm{trace} [(\nabla (a^l(\nabla u_{l,k})))^2]
        &=  N\lambda^4\mathrm{trace}(\nabla a^l(\nabla u_{l,k})\nabla^2u_{l,k}\nabla a^l(\nabla u_{l,k})\nabla^2u_{l,k})\\
        &\geq  \mathrm{trace}[\nabla a^l(\nabla u_{l,k})\nabla^2u_{l,k}(\nabla a^l(\nabla u_{l,k})\nabla^2u_{l,k})^T]\\
        &=|\nabla a^l(\nabla u_{l,k})\nabla^2u_{l,k}|^2\\
        &=|\nabla (a^l(\nabla u_{l,k}))|^2.
    \end{aligned}
\end{equation*}
As a consequence, we have
\begin{equation*}
    \begin{aligned}
        \sum\limits_{i,m=1}^{N}\partial_m (a^l_i(\nabla u_{l,k}))\partial_i(a^l_m(\nabla u_{l,k})\rho^2)
        &=\sum\limits_{i,m=1}^{N}\partial_m (a^l_i(\nabla u_{l,k}))\partial_i (a^l_m(\nabla u_{l,k}))\rho^2\\
        &+2\sum\limits_{i,m=1}^{N}\partial_m (a^l_i(\nabla u_{l,k}))a^l_m(\nabla u_{l,k})\rho_i\rho\\
        &\geq \frac{1}{N\lambda^4}|\nabla a^l(\nabla u_{l,k})|^2\rho^2+2\sum\limits_{i,m=1}^{N}\partial_m (a^l_i(\nabla u_{l,k}))a^l_m(\nabla u_{l,k})\rho_i\rho.
    \end{aligned}
\end{equation*}
Then, by Young's inequality, we obtain the following Caccioppoli type estimate:
\begin{eqnarray}\label{ineq:4.9}
    &&\quad \int_{\mathcal{C}_{k,R}}\left\lvert \nabla a^l(\nabla u_{l,k})\right\rvert^2\rho^2\mathrm{d}x \\
    \nonumber &&\leq C\int_{\mathcal{C}_{k,R}}\left\lvert a^l(\nabla u_{l,k})\right\rvert^2\left\lvert\nabla\rho \right\rvert^2 \mathrm{d}x+
     C \int_{\mathcal{C}_{k,R}}\left\lvert\nabla(e^u)\right\rvert\left\lvert a^l(\nabla u_{l,k})\right\rvert \rho^2\mathrm{d}x
\end{eqnarray}
for some constant $C=C(N,\lambda)$. Since $ \rho \in C_c^{\infty}(B_R(0)\setminus B_{\frac{1}{R}}(0))$, by approximation, we know that \eqref{ineq:4.9} holds for any  $\rho \in C_c^{\infty}(\mathbb{R}^N)$. Since $u_{l,k}\rightarrow u_k$ in $C^1_{loc}(\overline{\mathcal{C}_k}\setminus\{0\})$ and $a^l\rightarrow a$ uniformly, as $l\rightarrow+\infty$, we deduce that
\begin{equation}\label{ineq:4.10}
    \int_{\mathcal{C}_k}\left\lvert \nabla a(\nabla u_{k})\right\rvert^2\rho^2\mathrm{d}x\leq
    C\int_{\mathcal{C}_k}\left\lvert a(\nabla u_{k})\right\rvert^2\left\lvert\nabla\rho \right\rvert^2 \mathrm{d}x+
    C \int_{\mathcal{C}_k}\left\lvert\nabla(e^u)\right\rvert\left\lvert a(\nabla u_{k})\right\rvert \rho^2\mathrm{d}x.
\end{equation}
In particular, \eqref{ineq:4.10} implies $a(\nabla u_k)\in W^{1,2}_{loc}(\overline{\mathcal{C}_k})$ and $\{a(\nabla u_k)\}$
is uniformly bounded in $W^{1,2}_{loc}(\overline{\mathcal{C}_k})$. Hence, by letting $k\rightarrow+\infty$ in \eqref{ineq:4.10}, we deduce from \eqref{4+} that
\begin{equation*}
    \int_{\mathcal{C}}\left\lvert \nabla a(\nabla u)\right\rvert^2\rho^2\mathrm{d}x\leq
    C\int_{\mathcal{C}}\left\lvert a(\nabla u)\right\rvert^2\left\lvert\nabla\rho \right\rvert^2 \mathrm{d}x+
    C \int_{\mathcal{C}}\left\lvert\nabla(e^u)\right\rvert\left\lvert a(\nabla u)\right\rvert \rho^2\mathrm{d}x,
\end{equation*}
which implies immediately that $a(\nabla u)\in W^{1,2}_{loc}(\overline{\mathcal{C}})$.
\end{proof}

\section{Quantization results by the anisotropic isoperimetric inequality inside convex cones and the Pohozaev type identity}
In this section, based on the sharp asymptotic estimates on $u$ and $\nabla u$ and the second order regularity $a(\nabla u)\in W^{1,2}_{loc}(\overline{\mathcal{C}})$ for weak solution $u$, we will prove delicate quantization results on the total mass $\int_{\mathcal{C}}e^u\mathrm{d}x$ by using the anisotropic isoperimetric inequality inside convex cones and the Pohozaev type identity.  Note that the latter indicates  that the equality holds in the isoperimetric inequality inside convex cones for any weak solution $u$ to \eqref{eq:1.1}.

We denote $B_R^{\widehat{H}_0}:=B_R^{\widehat{H}_0}(0)$ for $R>0$. Let $\beta_0:=\left[\frac{\int_{\mathcal{C}}e^u \mathrm{d}x}{N\mathcal{L}^{N}(B_1^{\widehat{H}_0}\cap \mathcal{C})}\right]^{\frac{1}{N-1}}$, where $\mathcal{C}$ and $\widehat{H}_0$ are the same as in Theorem \ref{Th:1.1}. First, we can determine the precise value of the constant $\beta$ in the sharp asymptotic estimates on $u$ and $\nabla u$ through an approximation argument used in the proof of \cite[Theorem 1.3]{CRS}.
\begin{lem}\label{le:5.1}
    Let $\beta$ be the constant in the sharp asymptotic estimates on $u$ and $\nabla u$ in Proposition \ref{Prop:3.2}, then $\beta=\beta_0$.
\end{lem}
\begin{proof}
The main issue is that $u$ is not $C^1$-smooth up to the boundary $\partial\mathcal{C}$, we will argue by approximation arguments as in the proof of \cite[Theorem 1.3]{CRS}.

\smallskip

  Without loss of generality, we assume that the $N$-th direction vector $e_N$ belongs to $\mathcal{C}$ and there exists
  a convex function $g:\mathbb{R}^{N-1}\rightarrow\mathbb{R}$ such that
  \begin{equation*}
    \mathcal{C}=\left\{x_N>g(x_1,\cdots,x_{N-1})\right\}.
  \end{equation*}
  By the proof of \cite[Theorem 1.3]{CRS}, there exists a sequence of smooth convex functions $g_k:\mathbb{R}^{N-1}\rightarrow\mathbb{R}$ ($k=1,2,\cdots$) satisfying the following conditions:\\
  (1) $g_k>g_{k+1}$ in $\overline{B}$, where $B$ is a large ball $B\subset\mathbb{R}^{N-1}$ containing the projection of
  $\overline{B_R^{\widehat{H}_0}}$; \\
  (2) $g_k\rightarrow g$ uniformly in $\overline{B}$;\\
  (3) $\left\lvert \nabla g_k\right\rvert $ is bounded independently of $k$ and $\nabla g_k\rightarrow\nabla g$ a.e. in $\overline{B}$;\\
  (4) the sets $G_k$ defined by
  \begin{equation*}
    G_k:=\left\{x_N>g_k(x_1,\cdots,x_{N-1})\right\}\subseteq\mathcal{C}
  \end{equation*}
satisfy that $G_k\cap B_R^{\widehat{H}_0}$ are open sets with Lipschitz boundaries, $\partial G_k$ and
   $\partial B_R^{\widehat{H}_0}$ intersect transversally and $\{G_k\cap B_R^{\widehat{H}_0}\}_{k=1}^{+\infty}$ approximates $\mathcal{C}\cap B_R^{\widehat{H}_0}$ in $L^1$ sense.

\medskip

In view of $u\in C^1\left(\overline{G_k\cap B_R^{\widehat{H}_0}}\right) $, by choosing $\Omega=G_k\cap B_R^{\widehat{H}_0}$, $f(x)=e^{u(x)}$
   and ${\bf{a}} (x)=-a(\nabla u)$ in Lemma \ref{le:2.3}, we obtain that
   \begin{equation}\label{eq:5.1}
    \int_{\partial G_k\cap B_R^{\widehat{H}_0}}\left\langle a(\nabla u),-\nu\right\rangle \mathrm{d}\mathcal{H}^{N-1}+
    \int_{G_k\cap \partial B_R^{\widehat{H}_0}}\left\langle a(\nabla u),-\nu\right\rangle \mathrm{d}\mathcal{H}^{N-1}
    =  \int_{G_k\cap B_R^{\widehat{H}_0}}e^u \mathrm{d}x.
   \end{equation}
   It follows from \eqref{eq:1.1} and Proposition \ref{Proposition:4.1} that $\left\langle a(\nabla u),\nu\right\rangle=0$ a.e. on $\partial\mathcal{C}$. Since $\nabla u\in C^0(\mathcal{C}\cap B_R^{\widehat{H}_0})\cap L^{\infty}(\mathcal{C}\cap B_R^{\widehat{H}_0})$, then by the dominated
   convergence theorem and properties (2)--(3), we deduce that
   \begin{equation}\label{eq:5.2}
    \int_{\partial G_k\cap B_R^{\widehat{H}_0}}\left\langle a(\nabla u),-\nu\right\rangle \mathrm{d}\mathcal{H}^{N-1}
    \rightarrow \int_{\partial\mathcal{C}\cap B_R^{\widehat{H}_0}}\left\langle a(\nabla u),-\nu\right\rangle \mathrm{d}\mathcal{H}^{N-1}
    =0.
   \end{equation}
   By properties (1) and (2), we know that $G_k\cap \partial(B_R^{\widehat{H}_0})\rightarrow \mathcal{C}\cap \partial(B_R^{\widehat{H}_0})$, thus we obtain that
   \begin{equation}\label{eq:5.3}
    \int_{G_k\cap \partial B_R^{\widehat{H}_0}}\left\langle a(\nabla u),-\nu\right\rangle \mathrm{d}\mathcal{H}^{N-1}
    \rightarrow\int_{\mathcal{C}\cap \partial B_R^{\widehat{H}_0}}\left\langle a(\nabla u),-\nu\right\rangle \mathrm{d}\mathcal{H}^{N-1}.
   \end{equation}
   Moreover, since $G_k\cap B_R^{\widehat{H}_0}\rightarrow\mathcal{C}\cap B_R^{\widehat{H}_0}$ in $L^1$ sense, we know that
   \begin{equation}\label{eq:5.4}
    \int_{G_k\cap B_R^{\widehat{H}_0}}e^u \mathrm{d}x\rightarrow \int_{\mathcal{C}\cap B_R^{\widehat{H}_0}}e^u \mathrm{d}x.
   \end{equation}
Thus, by combining \eqref{eq:5.1}, \eqref{eq:5.2}, \eqref{eq:5.3} and \eqref{eq:5.4}, we deduce that
\begin{equation}\label{eq:5.5}
    \int_{\mathcal{C}\cap \partial B_R^{\widehat{H}_0}}\left\langle a(\nabla u),-\nu\right\rangle \mathrm{d}\mathcal{H}^{N-1}=
    \int_{\mathcal{C}\cap B_R^{\widehat{H}_0}}e^u \mathrm{d}x.
\end{equation}

By Proposition \ref{Prop:3.2}, we obtain that
\begin{equation*}
    \nabla u=-\beta \frac{\nabla \widehat{H}_0}{\widehat{H}_0}+o\left(\frac{1}{\widehat{H}_0}\right),
\end{equation*}
and
\begin{equation*}
    H(\nabla u)=\frac{\beta}{\widehat{H}_0}+o\left(\frac{1}{\widehat{H}_0}\right)
\end{equation*}
uniformly for $x\in\partial B_R^{\widehat{H}_0}$, as $R\rightarrow+\infty$. By \eqref{eq:5.5}, we deduce that
\begin{equation}\label{eq:5.6}
    \begin{aligned}
        \int_{\mathcal{C}\cap B_R^{\widehat{H}_0}}e^u \mathrm{d}x
            &=\int_{\mathcal{C}\cap \partial B_R^{\widehat{H}_0}}H^{N-1}(\nabla u)\left\langle \nabla H(\nabla u),-\nu\right\rangle \mathrm{d}\mathcal{H}^{N-1}\\
            &=\int_{\mathcal{C}\cap \partial B_R^{\widehat{H}_0}}H^{N-1}(\nabla u)\left\langle \nabla H(\nabla u),-\frac{\nabla \widehat{H}_0}{\lvert \nabla\widehat{H}_0\rvert} \right\rangle \mathrm{d}\mathcal{H}^{N-1}\\
            &=\frac{R}{\beta}\int_{\mathcal{C}\cap \partial B_R^{\widehat{H}_0}}H^{N-1}(\nabla u)\left\langle \nabla H(\nabla u),\frac{\nabla u+o(\frac{1}{R})}{\lvert \nabla\widehat{H}_0\rvert} \right\rangle \mathrm{d}\mathcal{H}^{N-1}.
    \end{aligned}
\end{equation}
Since $\left\langle \nabla H(\xi),\xi\right\rangle=H(\xi)$ due to the positive homogeneity of degree $1$ of $H$, for every $x\in\partial B_R^{\widehat{H}_0}$, we have
\begin{equation}\label{eq:5.7}
    \begin{aligned}
    \left\langle \nabla H(\nabla u),\frac{\nabla u+o(\frac{1}{R})}{\lvert \nabla\widehat{H}_0\rvert} \right\rangle
    &=\frac{1}{\lvert \nabla\widehat{H}_0\rvert}\left\langle \nabla H(\nabla u),\nabla u\right\rangle
    +\frac{\lvert \nabla H(\nabla u)\rvert }{\lvert \nabla\widehat{H}_0\rvert }o\left(\frac{1}{R}\right)\\
    &=\frac{1}{\lvert \nabla\widehat{H}_0\rvert} H(\nabla u)+\frac{\left\lvert \nabla H(\nabla u)\right\rvert }{\lvert \nabla\widehat{H}_0\rvert }o\left(\frac{1}{R}\right).
    \end{aligned}
\end{equation}
Moreover, recalling \eqref{norm0}, one has
\begin{equation*}
        \frac{1}{C_{H}}|x|\leq H_0(x)\leq \frac{1}{c_{H}}|x|, \qquad \forall \,\, x\in\mathbb{R}^{N}.
\end{equation*}
Thus by letting $x=\nabla H(\xi)$, we infer from $H_0(\nabla H(\xi))=1$ (see Lemma \ref{8,le:2.2}) that
\begin{equation}\label{lub}
    c_{H}\leq \left\lvert \nabla H(\xi)\right\rvert \leq C_{H}, \qquad \forall \,\, \xi\in\mathbb{R}^{N},
\end{equation}
where the constants $C_{H}:=\max\limits_{\xi\in \mathbb{S}^{N-1}}H(\xi)\geq c_{H}:=\min\limits_{\xi\in \mathbb{S}^{N-1}}H(\xi)>0$ and both are independent of $R$ and $x\in\partial B_R^{\widehat{H}_0}$. Similarly, by letting $\xi=\nabla \widehat{H}_0(x)$ in the inequality
$$c_{H}|\xi|=|\xi|\min\limits_{\mathbb{S}^{N-1}}\widehat{H}\leq \widehat{H}(\xi)\leq|\xi|\max\limits_{\mathbb{S}^{N-1}}\widehat{H}=C_{H}|\xi|, \qquad \forall \,\, \xi\in\mathbb{R}^{N},$$
we infer from $\widehat{H}(\nabla \widehat{H}_0(x))=1$ (see e.g. Lemma 2.3 in \cite{CFV}) that
 \begin{equation}\label{lub0}
    \frac{1}{C_{H}}\leq \lvert \nabla \widehat{H}_0(x)\rvert \leq \frac{1}{c_{H}},\qquad \forall \,\, x\in\mathbb{R}^{N},
 \end{equation}
where the positive constants $c_{H}$ and $C_{H}$ are independent of $R$ and $x\in\partial B_R^{\widehat{H}_0}$. From \eqref{lub} with $\xi=\nabla u$ and \eqref{lub0}, it follows that $c_{H}^{2}\leq\frac{\left\lvert \nabla H(\nabla u)\right\rvert }{\left\lvert \nabla\widehat{H}_0\right\rvert}\leq C_{H}^{2}$ is uniformly bounded, thus we can continue to simplify \eqref{eq:5.7} as follows:
 \begin{equation*}
    \begin{aligned}
        \left\langle \nabla H(\nabla u),\frac{\nabla u+o(\frac{1}{R})}{\lvert \nabla\widehat{H}_0\rvert} \right\rangle
        &=\frac{1}{\lvert \nabla\widehat{H}_0\rvert} H(\nabla u)+o\left(\frac{1}{R}\right)\\
        &=\frac{1}{\lvert \nabla\widehat{H}_0\rvert}H\left(-\beta \frac{\nabla \widehat{H}_0}{\widehat{H}_0}\right)+o\left(\frac{1}{R}\right)\\
        &=\frac{\beta}{R}\widehat{H}(\nu)+o\left(\frac{1}{R}\right)
    \end{aligned}
 \end{equation*}
uniformly for $x\in\partial B_R^{\widehat{H}_0}$, as $R\rightarrow+\infty$. Then by \eqref{eq:5.6}, we have
\begin{equation}\label{eq:5.8}
    \begin{aligned}
        \int_{\mathcal{C}\cap B_R^{\widehat{H}_0}}e^u \mathrm{d}x
            &=\frac{R}{\beta}\int_{\mathcal{C}\cap \partial B_R^{\widehat{H}_0}}H^{N-1}(\nabla u)\left\langle \nabla H(\nabla u),\frac{\nabla u+o\left(\frac{1}{R}\right)}{\lvert \nabla\widehat{H}_0\rvert} \right\rangle \mathrm{d}\mathcal{H}^{N-1}\\
            &=\frac{R}{\beta}\int_{\mathcal{C}\cap \partial B_R^{\widehat{H}_0}}\left(\frac{\beta}{R}+o\left(\frac{1}{R}\right)\right)^{N-1}\left(\frac{\beta}{R}\widehat{H}(\nu)+o\left(\frac{1}{R}\right)\right)\mathrm{d}\mathcal{H}^{N-1}  \\
            &=\frac{1}{R^{N-1}}\int_{\mathcal{C}\cap \partial B_R^{\widehat{H}_0}}\left(\beta^{N-1}\widehat{H}(\nu)+o_{R}(1)\right)\mathrm{d}\mathcal{H}^{N-1}.
    \end{aligned}
\end{equation}
Since
\begin{equation*}
    o\left(\frac{1}{R^{N-1}}\int_{\mathcal{C}\cap \partial B_R^{\widehat{H}_0}}\mathrm{d}\mathcal{H}^{N-1}\right)
    =o\left(\int_{\mathcal{C}\cap \partial B_1^{\widehat{H}_0}}\mathrm{d}\mathcal{H}^{N-1}\right)=o_{R}(1),
\end{equation*}
as $R\rightarrow+\infty$, and
\begin{equation*}
    \begin{aligned}
        \frac{\beta^{N-1}}{R^{N-1}}\int_{\mathcal{C}\cap \partial B_R^{\widehat{H}_0}}\widehat{H}(\nu)\mathrm{d}\mathcal{H}^{N-1}
        &=\beta^{N-1}\int_{\mathcal{C}\cap \partial B_1^{\widehat{H}_0}}\widehat{H}(\nu)\mathrm{d}\mathcal{H}^{N-1}\\
        &=\beta^{N-1}P_{\widehat{H}}(B_1^{\widehat{H}_0};\mathcal{C})\\
        &=\beta^{N-1}N\mathcal{L}^{N}(B_1^{\widehat{H}_0}\cap\mathcal{C}),
    \end{aligned}
\end{equation*}
then by \eqref{eq:5.8}, we have
\begin{equation*}
    \int_{\mathcal{C}\cap B_R^{\widehat{H}_0}}e^u \mathrm{d}x=\beta^{N-1}N\mathcal{L}^{N}(B_1^{\widehat{H}_0}\cap\mathcal{C})+o_{R}(1),\qquad \text{as}\,\, R\rightarrow+\infty.
\end{equation*}
Let $R\rightarrow+\infty$, we obtain that
\begin{equation*}
    \int_{\mathcal{C}}e^u \mathrm{d}x=\beta^{N-1}N\mathcal{L}^{N}(B_1^{\widehat{H}_0}\cap\mathcal{C}),
\end{equation*}
which implies $\beta=\beta_0$, thus we complete the proof.
\end{proof}

By using the the anisotropic isoperimetric inequality inside convex cones, we can derive the following a priori sharp lower bound estimate on the total mass $\int_{\mathcal{C}}e^u \mathrm{d}x$.
\begin{lem}\label{le:5.2}
    Let $u$ be a solution of \eqref{eq:1.1}, then
    \begin{equation}\label{ineq:5.9}
        \int_{\mathcal{C}}e^u \mathrm{d}x\geq c_N\mathcal{L}^{N}(\mathcal{C}\cap B_1^{\widehat{H}_0}),
    \end{equation}
    where $c_N=N(\frac{N^2}{N-1})^{N-1}$.
\end{lem}
\begin{proof}
  Since $u\in C^{1,\theta}_{loc}(\mathcal{C})$, then by \cite[Corollary 1.7]{ACF}, we obtain that
  $\mathcal{Z}_k:=\{x\in B_k(0)\cap\mathcal{C}:\nabla u(x)=0\}$ is a null set for all $k\in\mathbb{N}$, i.e., $\mathcal{L}^N(\mathcal{Z}_k)=0$ for
   all $k\in\mathbb{N}$. By $u\in C^{1,\theta}_{loc}(\mathcal{C})$, we deduce further that the set
   \begin{equation*}
    \{t\in\mathbb{R}: \exists x\in\mathcal{C}\ {\rm s.t.}\ u(x)=t, \nabla u(x)=0\}=\bigcup \limits_{k\in\mathbb{N}}u(\mathcal{Z}_k)
   \end{equation*}
   is a null set in $\mathbb{R}^1$. Moreover, $\Omega_t:=\{x\in\mathcal{C}:u(x)>t\}$ is a piecewise smooth set for a.e. $t\in(-\infty,t_0)$, where
   $t_0:=\sup\limits_{\mathcal{C}}u$, and $\Omega_t$ has a bounded Lebesgue measure for each $t\in(-\infty,t_0)$ in view
   of $\int_{\mathcal{C}}e^u\mathrm{d}x<+\infty$. By the divergence theorem in Lemma \ref{le:2.3}, we have
   \begin{equation}\label{eq:5.10}
    \begin{aligned}
        \int_{\Omega_t}e^u \mathrm{d}x
        &=-\int_{\partial\Omega_t}\left\langle a(\nabla u),\nu\right\rangle \mathrm{d}\mathcal{H}^{N-1}\\
        &=-\int_{\{x\in\partial\mathcal{C}:u(x)\geq t\}}\left\langle a(\nabla u),\nu\right\rangle \mathrm{d}\mathcal{H}^{N-1}-
        \int_{\{x\in\mathcal{C}:u(x)=t\}}\left\langle a(\nabla u),\nu\right\rangle \mathrm{d}\mathcal{H}^{N-1}\\
        &=- \int_{\{x\in\mathcal{C}:u(x)=t\}}\left\langle a(\nabla u),\nu\right\rangle \mathrm{d}\mathcal{H}^{N-1}.
    \end{aligned}
   \end{equation}
Note that $\nu=-\frac{\nabla u}{|\nabla u|}$ on $\{x\in\mathcal{C}:u(x)=t\}$ , then by \eqref{eq:5.10} and $\left\langle \nabla H(\xi),\xi\right\rangle=H(\xi)$, we have
\begin{equation*}
    \begin{aligned}
        \int_{\Omega_t}e^u \mathrm{d}x
        &=- \int_{\{x\in\mathcal{C}:u(x)=t\}}\left\langle a(\nabla u),\nu\right\rangle \mathrm{d}\mathcal{H}^{N-1}\\
        &=\int_{\{x\in\mathcal{C}:u(x)=t\}}H^{N-1}(\nabla u)\left\langle \nabla H(\nabla u),\frac{\nabla u}{|\nabla u|}\right\rangle \mathrm{d}\mathcal{H}^{N-1}\\
        &=\int_{\{x\in\mathcal{C}:u(x)=t\}}\frac{H^{N}(\nabla u)}{|\nabla u|}\mathrm{d}\mathcal{H}^{N-1}.
    \end{aligned}
\end{equation*}
Therefore, we obtain that
\begin{equation}\label{eq:5.11}
    \int_{\Omega_t}e^u \mathrm{d}x=\int_{\{x\in\mathcal{C}:u(x)=t\}}\frac{H^{N}(\nabla u)}{|\nabla u|}\mathrm{d}\mathcal{H}^{N-1}.
\end{equation}

\smallskip

For $\varepsilon\in(0,1)$, let $\Omega_{t,\varepsilon}:=\{x\in\mathcal{C}:u(x)>t, |\nabla u|>\varepsilon\}$,
 $\partial\mathcal{C}_t:=\{x\in\mathcal{C}:u(x)=t\}$. Then from the co-area formula, we deduce that
\begin{equation*}
    \mathcal{L}^N(\Omega_{t,\varepsilon})=\int_{t}^{t_0}\mathrm{d}s\int_{\partial\mathcal{C}_s\cap\{|\nabla u|>\varepsilon\}}\frac{1}{|\nabla u|}\mathrm{d}\mathcal{H}^{N-1}.
\end{equation*}
Since $\mathcal{L}^{N}(\Omega_t)<+\infty$, letting $\varepsilon\rightarrow0$ and by the monotone convergence theorem, we deduce that
\begin{equation*}
    \mathcal{L}^N(\Omega_t)=\int_{t}^{t_0}\mathrm{d}s\int_{\partial\mathcal{C}_s}\frac{1}{|\nabla u|}\mathrm{d}\mathcal{H}^{N-1} \quad\, \text{for\,\, a.e.}\,\, t<t_0.
\end{equation*}
Consequently, we have
\begin{equation}\label{eq:5.12}
    -\frac{d}{dt}\mathcal{L}^N(\Omega_t)=\int_{\partial\mathcal{C}_t}\frac{1}{|\nabla u|}\mathrm{d}\mathcal{H}^{N-1} \quad\, \text{for\,\, a.e.}\,\, t<t_0.
\end{equation}
Similarly, we have
\begin{equation*}
    \int_{\Omega_t}e^u\mathrm{d}x=\int_{t}^{t_0}\mathrm{d}s\int_{\partial\mathcal{C}_s}\frac{e^s}{|\nabla u|}\mathrm{d}\mathcal{H}^{N-1} \quad\, \text{for\,\, a.e.}\,\, t<t_0,
\end{equation*}
and hence, it holds
\begin{equation}\label{eq:5.13}
    -\frac{d}{dt}\left(\int_{\Omega_t}e^u\mathrm{d}x\right)=e^t\int_{\partial\mathcal{C}_t}\frac{1}{|\nabla u|}\mathrm{d}\mathcal{H}^{N-1} \quad\, \text{for\,\, a.e.}\,\, t<t_0.
\end{equation}
In view of \eqref{eq:5.11}, \eqref{eq:5.12} and \eqref{eq:5.13}, we deduce that
\begin{equation}\label{ineq:5.14}
    \begin{aligned}
        -\frac{d}{dt}\left(\int_{\Omega_t}e^u\mathrm{d}x\right)^{\frac{N}{N-1}}
        &=\frac{N}{N-1}\left(\int_{\partial\mathcal{C}_t}\frac{H^N(\nabla u)}{|\nabla u|}\mathrm{d}\mathcal{H}^{N-1}\right)^{\frac{1}{N-1}}
        e^t\int_{\partial\mathcal{C}_t}\frac{1}{|\nabla u|}\mathrm{d}\mathcal{H}^{N-1}\\
        &\geq  \frac{N}{N-1}e^t \left( \int_{\partial\mathcal{C}_t}\frac{H(\nabla u)}{|\nabla u|}\mathrm{d}\mathcal{H}^{N-1}\right)^{\frac{N}{N-1}}\\
        &=\frac{N}{N-1}e^t \left( \int_{\partial\mathcal{C}_t}H(-\nu)\mathrm{d}\mathcal{H}^{N-1}\right)^{\frac{N}{N-1}}\\
        &=\frac{N}{N-1}e^t \left( \int_{\partial\mathcal{C}_t}\widehat{H}(\nu)\mathrm{d}\mathcal{H}^{N-1}\right)^{\frac{N}{N-1}}.
    \end{aligned}
\end{equation}
Since $\int_{\partial\mathcal{C}_t}\widehat{H}(\nu)\mathrm{d}\mathcal{H}^{N-1}=\int_{\partial\Omega_t\cap\mathcal{C}}\widehat{H}(\nu)\mathrm{d}\mathcal{H}^{N-1}$, by Theorem \ref{thm1}, we have
\begin{equation}\label{ineq:5.15}
    \begin{aligned}
        \int_{\partial\mathcal{C}_t}\widehat{H}(\nu)\mathrm{d}\mathcal{H}^{N-1}&=P_{\widehat{H}}(\Omega_t;\mathcal{C})\\
        &\geq \left(\frac{\mathcal{L}^N(\Omega_t\cap\mathcal{C})}{\mathcal{L}^N(B_1^{\widehat{H}_0}\cap\mathcal{C})}\right)
        ^{\frac{N-1}{N}} P_{\widehat{H}}(B_1^{\widehat{H}_0};\mathcal{C})\\
        &=N\mathcal{L}^N(B_1^{\widehat{H}_0}\cap\mathcal{C})\left(\frac{\mathcal{L}^N(\Omega_t\cap\mathcal{C})}{\mathcal{L}^N(B_1^{\widehat{H}_0}\cap\mathcal{C})}\right)
        ^{\frac{N-1}{N}},
    \end{aligned}
\end{equation}
where we have used $ P_{\widehat{H}}(B_1^{\widehat{H}_0};\mathcal{C})=N\mathcal{L}^N(B_1^{\widehat{H}_0}\cap\mathcal{C})$,
see \cite[Formula (1.14)]{CRS}. Combining \eqref{ineq:5.14} with \eqref{ineq:5.15}, we deduce that
\begin{equation}\label{ineq:5.16}
    \begin{aligned}
        -\frac{d}{dt}\left(\int_{\Omega_t}e^u\mathrm{d}x\right)^{\frac{N}{N-1}}
        &\geq\frac{N}{N-1}e^t \left( \int_{\partial\mathcal{C}_t}\widehat{H}(\nu)\mathrm{d}\mathcal{H}^{N-1}\right)^{\frac{N}{N-1}}\\
        &\geq\frac{N}{N-1}e^t N^{\frac{N}{N-1}}\frac{\left(\mathcal{L}^N(B_1^{\widehat{H}_0}\cap\mathcal{C})\right)^{\frac{N}{N-1}} }{\mathcal{L}^N(B_1^{\widehat{H}_0}\cap\mathcal{C})}
        \mathcal{L}^N(\Omega_t\cap\mathcal{C})\\
        &=\frac{N^2}{N-1}e^t N^{\frac{1}{N-1}}\left(\mathcal{L}^N(B_1^{\widehat{H}_0}\cap\mathcal{C})\right)^{\frac{1}{N-1}}
        \mathcal{L}^N(\Omega_t\cap\mathcal{C}).
    \end{aligned}
\end{equation}
Letting $S_N:=\frac{N^2}{N-1}N^{\frac{1}{N-1}}\left(\mathcal{L}^N(B_1^{\widehat{H}_0}\cap\mathcal{C})\right)^{\frac{1}{N-1}}$, and
then integrating \eqref{ineq:5.16} from $-\infty$ to $t_0$, we deduce that
\begin{equation*}
    \begin{aligned}
        \left(\int_{\mathcal{C}}e^u\mathrm{d}x\right)^{\frac{N}{N-1}}
        &\geq S_N\int_{-\infty}^{t_0}e^t\mathrm{d}t\int_{\mathbb{R}^N}\chi _{\Omega_t\cap\mathcal{C}}\mathrm{d}x\\
        &= S_N \int_{\mathbb{R}^N}\int_{-\infty}^{t_0}e^t\chi_{\Omega_t\cap\mathcal{C}}\mathrm{d}t\mathrm{d}x\\
        &=S_N \int_{\mathcal{C}}e^u\mathrm{d}x.
    \end{aligned}
\end{equation*}
Thus we obtain \eqref{ineq:5.9} and hence complete the proof.
\end{proof}

Furthermore, by the Pohozaev type identity and Proposition \ref{Prop:3.2}, we will prove the following quantization result on the total mass $\int_{\mathcal{C}}e^{u}\mathrm{d}x$, and hence the equality in Lemma \ref{le:5.2} holds for any weak solution of \eqref{eq:1.1}.
\begin{prop}\label{Prop:5.3}
    Let $u$ be a solution of \eqref{eq:1.1}, then
    \begin{equation}\label{eq:5.17}
        \int_{\mathcal{C}}e^u \mathrm{d}x= c_N\mathcal{L}^{N}(\mathcal{C}\cap B_1^{\widehat{H}_0}),
    \end{equation}
    where $c_N=N\left(\frac{N^2}{N-1}\right)^{N-1}$.
\end{prop}
\begin{proof}
Let $\Omega\subset\mathbb{R}^{N}$ be a bounded open set with Lipschitz boundary and $u\in C^{1}(\overline{\Omega})$ solves $-\Delta^{H}_{N}u=f$ in $\Omega$, where $f\in L^{1}(\Omega)\cap L^{\infty}_{loc}(\Omega)$, then the following Pohozaev identity (c.f. \cite[Lemma 4.2]{CL2} in the case $p=N$, see also \cite[Lemma 3.3]{E}) holds for any $y\in\mathbb{R}^{N}$:
\begin{equation}\label{Poho}
\begin{aligned}
& -\int_{\Omega} f(x)\langle x-y, \nabla u\rangle \mathrm{d}x \\
= & \int_{\partial \Omega}\left[H^{N-1}(\nabla u)\langle\nabla H(\nabla u), \nu\rangle\langle x-y, \nabla u\rangle-\frac{H^N(\nabla u)}{N}\langle x-y, \nu\rangle\right] \mathrm{d}\mathcal{H}^{N-1},
\end{aligned}
\end{equation}
where $\nu$ is the unit outer normal vector to $\partial\Omega$. For any $R>0$, since $u$ may not belong to $C^1\left(\overline{\mathcal{C}\cap B_R^{\widehat{H}_0}( 0)}\right)$, in order to apply the Pohozaev identity, we need to approximate $\mathcal{C}\cap B_R^{\widehat{H}_0}(0)$ by a sequence of Lipschitz domains $\{G_k\cap B_R^{\widehat{H}_0}(0)\}_{k=1}^{+\infty}$ as in the proof of Lemma \ref{le:5.1}. Due to $u\in C^1\left(\overline{G_k\cap B_R^{\widehat{H}_0}( 0)}\right) $, we can choose $\Omega=G_k\cap B_R^{\widehat{H}_0}(0)$, $y=0$ and $f=e^u$ in the Pohozaev identity \eqref{Poho} and obtain that
     \begin{equation}\label{eq:5.18}
        \begin{aligned}
            &\quad-\int_{G_k\cap B_R^{\widehat{H}_0}( 0)}e^u\left\langle x ,\nabla u\right\rangle \mathrm{d}x\\
            &=\int_{\partial(G_k\cap B_R^{\widehat{H}_0}( 0))} \left[H^{N-1}(\nabla u)
            \left\langle \nabla H(\nabla u),\nu\right\rangle
            \left\langle x ,\nabla u\right\rangle-\frac{H^{N}(\nabla u)}{N}\left\langle x ,\nu\right\rangle \right]\mathrm{d}\mathcal{H}^{N-1}.
        \end{aligned}
     \end{equation}
By the divergence theorem, we have
\begin{equation}\label{eq:5.19}
        -\int_{G_k\cap B_R^{\widehat{H}_0}( 0)}e^u\left\langle x ,\nabla u\right\rangle \mathrm{d}x
        =N\int_{G_k\cap B_R^{\widehat{H}_0}( 0)}e^u \mathrm{d}x
        -\int_{\partial(G_k\cap B_R^{\widehat{H}_0}( 0))}e^u\left\langle x ,\nu\right\rangle \mathrm{d}\mathcal{H}^{N-1}.
\end{equation}
Note that
\begin{equation*}
    \begin{aligned}
        \int_{\partial(G_k\cap B_R^{\widehat{H}_0}( 0))}e^u\left\langle x ,\nu\right\rangle \mathrm{d}\mathcal{H}^{N-1}
        &=\int_{\partial G_k\cap B_R^{\widehat{H}_0}( 0)}e^u\left\langle x ,\nu\right\rangle \mathrm{d}\mathcal{H}^{N-1}\\
        &\quad+\int_{G_k\cap \partial B_R^{\widehat{H}_0}( 0)}e^u\left\langle x ,\nu\right\rangle \mathrm{d}\mathcal{H}^{N-1}.
    \end{aligned}
\end{equation*}
Since $u\in C^1\left(\overline{G_k\cap B_R^{\widehat{H}_0}( 0)}\right) $, by letting $k\rightarrow+\infty$ in \eqref{eq:5.19},
we obtain that
\begin{equation*}
    \begin{aligned}
        -\int_{\mathcal{C}\cap B_R^{\widehat{H}_0}( 0)}e^u\left\langle x ,\nabla u\right\rangle \mathrm{d}x
        &=N\int_{\mathcal{C}\cap B_R^{\widehat{H}_0}( 0)}e^u \mathrm{d}x
        -\int_{\partial\mathcal{C}\cap B_R^{\widehat{H}_0}( 0)}e^u\left\langle x ,\nu\right\rangle \mathrm{d}\mathcal{H}^{N-1}\\
        &\quad-\int_{\mathcal{C}\cap \partial B_R^{\widehat{H}_0}( 0)}e^u\left\langle x ,\nu\right\rangle \mathrm{d}\mathcal{H}^{N-1}.
    \end{aligned}
\end{equation*}
Since $\mathcal{C}=\mathbb{R}^{k}\times\mathcal{\widetilde{C}}$ is a convex cone, where $k\in\{0,\cdots ,N\}$ and $\mathcal{\widetilde{C}}\subset\mathbb{R}^{N-k}$ is an open convex cone with vertex at the origin $0_{\mathbb{R}^{N-k}}$, so $\left\langle x ,\nu\right\rangle =0$ for any $x\in\partial\mathcal{C}$. Then we derive that
\begin{equation}\label{eq:5.20}
    -\int_{\mathcal{C}\cap B_R^{\widehat{H}_0}( 0)}e^u\left\langle x ,\nabla u\right\rangle \mathrm{d}x
    =N\int_{\mathcal{C}\cap B_R^{\widehat{H}_0}( 0)}e^u \mathrm{d}x
    -\int_{\mathcal{C}\cap \partial B_R^{\widehat{H}_0}( 0)}e^u\left\langle x ,\nu\right\rangle \mathrm{d}\mathcal{H}^{N-1}.
\end{equation}
Similarly, by letting $k\rightarrow+\infty$ in \eqref{eq:5.18}, we deduce that
\begin{equation}\label{eq:5.21}
    \begin{aligned}
        &\quad-\int_{\mathcal{C}\cap B_R^{\widehat{H}_0}( 0)}e^u\left\langle x ,\nabla u\right\rangle \mathrm{d}x\\
        &=\int_{\mathcal{C}\cap \partial B_R^{\widehat{H}_0}( 0)} \left[H^{N-1}(\nabla u)
        \left\langle \nabla H(\nabla u),\nu\right\rangle
        \left\langle x ,\nabla u\right\rangle-\frac{H^{N}(\nabla u)}{N}\left\langle x ,\nu\right\rangle \right]\mathrm{d}\mathcal{H}^{N-1},
    \end{aligned}
\end{equation}
where we have used the fact that $a(\nabla u)\cdot \nu=0$ on $\partial\mathcal{C}$ and $\left\langle x ,\nu\right\rangle =0$  for any $x\in\partial\mathcal{C}$. By combining \eqref{eq:5.20} with \eqref{eq:5.21}, we arrive at
\begin{equation}\label{eq:5.22}
    \begin{aligned}
        &\quad N\int_{\mathcal{C}\cap B_R^{\widehat{H}_0}( 0)}e^u \mathrm{d}x
        -\int_{\mathcal{C}\cap \partial B_R^{\widehat{H}_0}( 0)}e^u\left\langle x ,\nu\right\rangle \mathrm{d}\mathcal{H}^{N-1}\\
        & =\int_{\mathcal{C}\cap \partial B_R^{\widehat{H}_0}( 0)} \left[H^{N-1}(\nabla u)
        \left\langle \nabla H(\nabla u),\nu\right\rangle
        \left\langle x ,\nabla u\right\rangle-\frac{H^{N}(\nabla u)}{N}\left\langle x ,\nu\right\rangle \right]\mathrm{d}\mathcal{H}^{N-1}.
    \end{aligned}
\end{equation}

Note that $\nu=\frac{\nabla \widehat{H}_0(x)}{\left\lvert \nabla\widehat{H}_0(x)\right\rvert}$ for $x\in\partial B_R^{\widehat{H}_0}( 0)$, by the computation in the proof of Lemma \ref{le:5.1}, we know that $\left\langle \nabla H(\nabla u),\nu\right\rangle= -\widehat{H}(\nu)+o_{R}(1)$ uniformly for $x\in\partial B_R^{\widehat{H}_0}( 0)$, as $R\rightarrow+\infty$. Since $\nabla u=-\beta_{0}\frac{\nabla \widehat{H}_0}{R}+o(\frac{1}{R})$ uniformly for $x\in\partial B_R^{\widehat{H}_0}( 0)$, as $R\rightarrow+\infty$, thus
  \begin{equation*}
    \begin{aligned}
        \left\langle x,\nabla u\right\rangle
        &=  \left\langle x,-\beta_{0}\frac{\nabla \widehat{H}_0}{R}+o\left(\frac{1}{R}\right)\right\rangle
        =-\frac{\beta_0}{R}\langle x,\nabla \widehat{H}_0\rangle+o_{R}(1)\\
        &=-\frac{\beta_0}{R}\widehat{H}_0(x)+o_{R}(1)=-\beta_0+o_{R}(1)
    \end{aligned}
  \end{equation*}
  uniformly for $x\in\partial B_R^{\widehat{H}_0}( 0)$, as $R\rightarrow+\infty$. Moreover, since $H(\nabla u)=\frac{\beta_0}{R}+o(\frac{1}{R})$
  uniformly for $x\in\partial B_R^{\widehat{H}_0}( 0)$, as $R\rightarrow+\infty$, we have
  \begin{equation}\label{eq:5.23}
    \begin{aligned}
        &\quad\int_{\mathcal{C}\cap \partial B_R^{\widehat{H}_0}( 0)} H^{N-1}(\nabla u)
        \left\langle \nabla H(\nabla u),\nu\right\rangle
        \left\langle x ,\nabla u\right\rangle \mathrm{d}\mathcal{H}^{N-1}\\
        &=\int_{\mathcal{C}\cap \partial B_R^{\widehat{H}_0}( 0)}\left(\frac{\beta_0}{R}+o\left(\frac{1}{R}\right)\right)^{N-1}\left(-\widehat{H}(\nu)+o_{R}(1)\right)
         \left(-\beta_0+o_{R}(1)\right)\mathrm{d}\mathcal{H}^{N-1} \\
         &=\frac{\beta_0^{N}}{R^{N-1}} \int_{\mathcal{C}\cap \partial B_R^{\widehat{H}_0}(0)}\left(\widehat{H}(\nu)+o_{R}(1)\right)\mathrm{d}\mathcal{H}^{N-1}\\
         &=  \beta_0^{N}\int_{\mathcal{C}\cap \partial B_1^{\widehat{H}_0}(0)}\widehat{H}(\nu)\mathrm{d}\mathcal{H}^{N-1}+o_{R}(1)\\
         &=\beta_0^{N}P_{\widehat{H}} \left(B_1^{\widehat{H}_0}(0);\mathcal{C}\right) +o_{R}(1)\\
         &=N\beta_0^{N}\mathcal{L}^N(B_1^{\widehat{H}_0}(0)\cap\mathcal{C})+o_{R}(1)
    \end{aligned}
  \end{equation}
  uniformly for $x\in\partial B_R^{\widehat{H}_0}( 0)$, as $R\rightarrow+\infty$. Since $\nu=\frac{\nabla\widehat{H}_0(x)}{\left\lvert\nabla\widehat{H}_0(x)\right\rvert}$ for $x\in\partial B_R^{\widehat{H}_0}( 0)$ and $\langle x,\nabla \widehat{H}_0(x)\rangle =\widehat{H}_0(x)$, thus $\left\langle x,\nu\right\rangle=\frac{R}{\left\lvert\nabla\widehat{H}_0 \right\rvert}$ on $\partial B_R^{\widehat{H}_0}( 0)$.  Then by co-area formula, we have
   \begin{equation}\label{eq:5.24}
    \begin{aligned}
        &\quad\int_{\mathcal{C}\cap \partial B_R^{\widehat{H}_0}( 0)} \frac{H^{N}(\nabla u)}{N}\left\langle x ,\nu\right\rangle \mathrm{d}\mathcal{H}^{N-1}\\
        &=\frac{1}{N}\left(\frac{\beta_0^N}{R^{N-1}}+o\left(\frac{1}{R^{N-1}}\right)\right)
        \int_{\mathcal{C}\cap \partial B_R^{\widehat{H}_0}( 0)}
        \frac{1}{\lvert\nabla\widehat{H}_0 \rvert }\mathrm{d}\mathcal{H}^{N-1} \\
        &=\frac{1}{N}\left(\frac{\beta_0^N}{R^{N-1}}+o\left(\frac{1}{R^{N-1}}\right)\right)
        \frac{\mathrm{d}}{\mathrm{d}t}\left(\int_{\mathcal{C}\cap B_t^{\widehat{H}_0}( 0)}\mathrm{d}x\right) \bigg|_{t=R}\\
        &=  \frac{1}{N}\left(\frac{\beta_0^N}{R^{N-1}}+o\left(\frac{1}{R^{N-1}}\right)\right)
        \frac{\mathrm{d}}{\mathrm{d}t}\left(t^N\right) \bigg|_{t=R}\cdot\mathcal{L}^N(B_1^{\widehat{H}_0}\cap\mathcal{C})\\
        &=(\beta_0^N+o_{R}(1))\mathcal{L}^N(B_1^{\widehat{H}_0}\cap\mathcal{C})
    \end{aligned}
   \end{equation}
   uniformly for $x\in\partial B_R^{\widehat{H}_0}( 0)$, as $R\rightarrow+\infty$. Moreover, by Proposition \ref{Prop:3.2} and Lemma \ref{le:5.1}, we obtain that
   \begin{equation*}
    e^{u(x)}\leq C\left[\widehat{H}_0(x)\right]^{-\beta_0}
   \end{equation*}
   for some constant $C>0$, as $|x|\rightarrow+\infty$ with $x\in\mathcal{C}$. Finally, it follows from Lemma \ref{le:5.2} that
   $$\beta_0=\left[\frac{\int_{\mathcal{C}}e^u \mathrm{d}x}{N\mathcal{L}^{N}(B_1^{\widehat{H}_0}\cap \mathcal{C})}\right]^{\frac{1}{N-1}}\geq (c_N)^{\frac{1}{N-1}}\geq \frac{N^2}{N-1}>N.$$
   Thus we can infer from $|x|\leq C\widehat{H}_0(x)$ that
   \begin{equation}\label{ineq:5.25}
    \begin{aligned}
        \int_{\mathcal{C}\cap \partial B_R^{\widehat{H}_0}( 0)}e^u\left\lvert \left\langle x ,\nu\right\rangle \right\rvert \mathrm{d}\mathcal{H}^{N-1}
        &\leq C    \int_{\mathcal{C}\cap \partial B_R^{\widehat{H}_0}( 0)}e^u\widehat{H}_0(x)\mathrm{d}\mathcal{H}^{N-1}\\
        &\leq C R^{N-\beta_0}\rightarrow0,\qquad\,\,\,\,\,\, \text{as}\ R\rightarrow+\infty.
    \end{aligned}
   \end{equation}
   By letting $R\rightarrow+\infty$ in \eqref{eq:5.22}, combining \eqref{eq:5.23} with \eqref{eq:5.24} and \eqref{ineq:5.25},
   we deduce that
   \begin{equation*}
    N\int_{\mathcal{C}}e^u \mathrm{d}x=(N-1)\beta_0^N\mathcal{L}^N(B_1^{\widehat{H}_0}\cap\mathcal{C}),
   \end{equation*}
   this concludes our proof of Proposition \ref{Prop:5.3}.
\end{proof}

\section{Proof of our main Theorem 1.1}
\noindent {\bf Proof of Theorem 1.1.} By Lemma \ref{le:5.2} and Proposition \ref{Prop:5.3}, we deduce that the inequalities \eqref{ineq:5.14}
and \eqref{ineq:5.15} are actually equalities. Thus by the characteristics of equalities of the anisotropic isoperimetric inequality and the H$\ddot{\text{o}}$lder inequality, we obtain that, for a.e. $t\leq t_0$, \\
$\bullet $  \quad $\Omega_t=\mathcal{C}\cap B_{R(t)}^{\widehat{H}_0}(x(t))$ for some $R(t)>0$ and $x(t)\in\mathbb{R}^k\times\{0_{\mathbb{R}^{N-k}}\}$,
then $\partial\mathcal{C}_t=\partial B_{R(t)}^{\widehat{H}_0}(x(t))\cap\mathcal{C}$; \\
$\bullet $ \quad  $\frac{H^N(\nabla u)}{\left\lvert \nabla u\right\rvert }$ is a multiple of $\frac{1}{\left\lvert \nabla u\right\rvert}$ on
$\partial\mathcal{C}_t$, then by the continuity of $\nabla u$, we have $H(\nabla u)$ is a positive constant on $\partial\mathcal{C}_t$.

\medskip

Let $M(t):=\int_{\Omega_t}e^u\mathrm{d}x$, then by \eqref{eq:5.11}, we have
\begin{equation}\label{eq:6.1}
    \begin{aligned}
        M(t)
        &=\int_{\partial B_{R(t)}^{\widehat{H}_0}(x(t))\cap\mathcal{C}}\frac{H^N(\nabla u)}{\left\lvert \nabla u\right\rvert }\mathrm{d}\mathcal{H}^{N-1}\\
        &=H^{N-1}(\nabla u)\int_{\partial B_{R(t)}^{\widehat{H}_0}(x(t))\cap\mathcal{C}}\widehat{H}(\nu)\mathrm{d}\mathcal{H}^{N-1}\\
        &=NH^{N-1}(\nabla u)\mathcal{L}^{N}(B_{1}^{\widehat{H}_0}(0)\cap\mathcal{C})R^{N-1}(t).
    \end{aligned}
\end{equation}
Here we have used the fact that $x(t)\in\mathbb{R}^k\times\{0_{\mathbb{R}^{N-k}}\}$, $\mathcal{C}=\mathbb{R}^k\times\widetilde{\mathcal{C}} $ and
$\widehat{H}$ is a positively homogenous function of degree one. Since the inequalities in \eqref{ineq:5.16} are actually equalities, we have
\begin{equation}\label{eq:6.2}
    \begin{aligned}
        \frac{d}{dt}M^{\frac{N}{N-1}}(t)&=\frac{N}{N-1}M^{\frac{1}{N-1}}M^{\prime}(t)\\
        &=-\frac{N^2e^t}{N-1}N^{\frac{1}{N-1}}\left[\mathcal{L}^{N}(B_{1}^{\widehat{H}_0}(0)\cap\mathcal{C})\right]^{\frac{N}{N-1}}R^N(t)
    \end{aligned}
\end{equation}
for a.e. $t<t_0$.

\smallskip

By letting $\Omega=\Omega_t$, $f=e^u$ and $y=x(t)$ in the Pohozaev identity \eqref{Poho} and  using the same approximation method in the proof of Proposition \ref{Prop:5.3}, we can derive that
 \begin{equation}\label{eq:6.3}
    \begin{aligned}
        &\quad N\int_{\mathcal{C}\cap B_{R(t)}^{\widehat{H}_0}(x(t))}e^u \mathrm{d}x
        -\int_{\mathcal{C}\cap \partial B_{R(t)}^{\widehat{H}_0}(x(t))}e^u\left\langle x-x(t) ,\nu\right\rangle \mathrm{d}\mathcal{H}^{N-1}\\
        & =\int_{\mathcal{C}\cap \partial B_{R(t)}^{\widehat{H}_0}(x(t))} \bigg[H^{N-1}(\nabla u)
        \left\langle \nabla H(\nabla u),\nu\right\rangle
        \left\langle x-x(t) ,\nabla u\right\rangle \\
        &\qquad\qquad\qquad\qquad -\frac{H^{N}(\nabla u)}{N}\left\langle x-x(t) ,\nu\right\rangle \bigg]\mathrm{d}\mathcal{H}^{N-1}.
    \end{aligned}
\end{equation}
Noting that $\nu=-\frac{\nabla u}{\left\lvert \nabla u\right\rvert }=\frac{\nabla \widehat{H}_0(x-x(t))}{\left\lvert \nabla \widehat{H}_0(x-x(t))\right\rvert}$
on $\partial\mathcal{C}_t$ and $x-x(t)=R(t)\nu$ on $\partial\mathcal{C}_t$, we deduce that
$\left\langle x-x(t),\nu\right\rangle=R(t)$, and
\begin{equation}\label{eq:6.4}
    \begin{aligned}
        \left\langle \nabla H(\nabla u),\nu\right\rangle
        \left\langle x-x(t) ,\nabla u\right\rangle&=\left\langle \nabla H(\nabla u),-\frac{\nabla u}{\left\lvert \nabla u\right\rvert }\right\rangle
        \left\langle -\frac{R(t)\nabla u}{\left\lvert \nabla u\right\rvert } ,\nabla u\right\rangle \\
        &=H(\nabla u)R(t).
    \end{aligned}
\end{equation}
Then from \eqref{eq:6.3}, we infer that
\begin{equation}\label{eq:6.5}
    \begin{aligned}
        N\int_{\Omega_t}e^u \mathrm{d}x
        -e^tR(t)\int_{\mathcal{C}\cap \partial B_{R(t)}^{\widehat{H}_0}(x(t))}\mathrm{d}\mathcal{H}^{N-1}=\frac{N-1}{N}H^N(\nabla u)R(t)\int_{\mathcal{C}\cap \partial B_{R(t)}^{\widehat{H}_0}(x(t))}\mathrm{d}\mathcal{H}^{N-1}.
    \end{aligned}
\end{equation}
From \eqref{eq:6.5}, we get
\begin{equation}\label{eq:6.6}
    M(t)=e^tR^{N}(t)\mathcal{L}^{N}(B_1^{\widehat{H}_0}(0)\cap\mathcal{C})+\frac{N-1}{N}H^{N}(\nabla u)R^N(t)\mathcal{L}^{N}(B_1^{\widehat{H}_0}(0)\cap\mathcal{C}).
\end{equation}
Then \eqref{eq:6.1} and \eqref{eq:6.6} yield that
\begin{equation}\label{eq:6.7}
    e^t R^N(t)=\frac{M(t)}{\mathcal{L}^{N}(B_1^{\widehat{H}_0}(0)\cap\mathcal{C})}-\frac{N-1}{N}\left[\frac{M(t)}{N\mathcal{L}^{N}(B_1^{\widehat{H}_0}(0)\cap\mathcal{C})}\right] ^{\frac{N}{N-1}}.
\end{equation}
By \eqref{eq:6.2}, we have
\begin{equation}\label{aeq:6.8}
    M^{\prime}(t)=-N^{\frac{N}{N-1}}e^tR^N(t)\left[\mathcal{L}^{N}(B_1^{\widehat{H}_0}(0)\cap\mathcal{C})\right]^{\frac{N}{N-1}}M^{-\frac{1}{N-1}}(t).
\end{equation}
Inserting \eqref{eq:6.7} into \eqref{aeq:6.8}, we get
\begin{equation*}
    \begin{aligned}
        M^{\prime}(t)=&-N^{\frac{N}{N-1}}\left(\frac{M(t)}{\mathcal{L}^{N}(B_1^{\widehat{H}_0}(0)\cap\mathcal{C})}-\frac{N-1}{N}\left[\frac{M(t)}{N\mathcal{L}^{N}(B_1^{\widehat{H}_0}(0)\cap\mathcal{C})}\right] ^{\frac{N}{N-1}}\right)\\
        &\cdot\left[\mathcal{L}^{N}(B_1^{\widehat{H}_0}(0)\cap\mathcal{C})\right]^{\frac{N}{N-1}}M^{-\frac{1}{N-1}}(t).
    \end{aligned}
\end{equation*}
Thus by  direct calculations, we can obtain that
\begin{equation}\label{eq:6.8}
    M^{\prime}(t)=\frac{N-1}{N}M(t)-N^{\frac{N}{N-1}}\left[\mathcal{L}^{N}(B_1^{\widehat{H}_0}(0)\cap\mathcal{C})\right]^{\frac{1}{N-1}}M^{\frac{N-2}{N-1}}(t).
\end{equation}
Let $B_N:=\frac{N}{N-1}N^{\frac{N}{N-1}}\left[\mathcal{L}^{N}(B_1^{\widehat{H}_0}(0)\cap\mathcal{C})\right]^{\frac{1}{N-1}}$. Due to
\begin{equation*}
    \int\frac{\mathrm{d}M}{\frac{N-1}{N}M-\frac{N-1}{N}B_N M^{\frac{N-2}{N-1}}}
    =N\text{ln}\lvert M^{\frac{1}{N-1}}-B_N \rvert
\end{equation*}
and $M(t_0)=0$, we can integrate \eqref{eq:6.8} from $t$ to $t_0$ and get
$$\lvert M^{\frac{1}{N-1}}-B_N \rvert=B_Ne^{\frac{t-t_0}{N}}.$$
Moreover, noting that $M(t)$ is decreasing with respect to $t$, we deduce that
\begin{equation}\label{eq:6.9}
    M(t)=\left[B_N\left(1-e^{\frac{t-t_0}{N}}\right) \right]^{N-1}.
\end{equation}
Inserting \eqref{eq:6.9} into \eqref{eq:6.7}, we get
\begin{equation}\label{eq:6.10}
    R^N(t)=N\left(\frac{N^2}{N-1}\right)^{N-1}\left(1-e^{\frac{t-t_0}{N}}\right)^{N-1} e^{-\frac{(N-1)t+t_0}{N}},
\end{equation}
for a.e. $t\leq t_0$. Since \eqref{eq:6.10} implies that $R(t)$ is monotone and continuous, thus \eqref{eq:6.10} can be rewritten as
\begin{equation}\label{eq:6.11}
    e^t=\frac{c_N\lambda^N}{\left[1+\lambda^{\frac{N}{N-1}}R^{\frac{N}{N-1}}(t)\right]^N},
\end{equation}
where
\begin{equation*}
    \lambda:=\left[\frac{e^{t_0}}{N\left(\frac{N^2}{N-1}\right)^{N-1} }\right]^{\frac{1}{N}},\quad t=u(x),\quad \widehat{H}_0(x-x(t))=R(t),\quad x\in\partial\mathcal{C}_{t}.
\end{equation*}

\smallskip

Now we only need to show that there exists $x_0\in\mathbb{R}^k\times\{0_{\mathbb{R}^{N-k}}\}$ such that $x(t)\equiv x_0$ for $t<t_0$. First
notice that if $s<t$, then $B_{R(t)}^{\widehat{H}_0}(x(t))\cap\mathcal{C}=\Omega_t\subset\Omega_s=B_{R(s)}^{\widehat{H}_0}(x(s))\cap\mathcal{C}$.  Thus
\begin{equation*}
    \widehat{H}_0(x(t)-x(s))\leq C\left\lvert R(t)-R(s)\right\rvert
\end{equation*}
for any $s,t<t_0$ and some constant $C>0$ depending only on $\widehat{H}_0$. Therefore we know that $x(t)$ is locally Lipschitz
 continuous on $(-\infty, t_0)$. Assume that there exists $t_1<t_0$ such that $x_1:=x^{\prime}(t_1)\neq0$. For $t<t_0$ and
 $\omega\in\partial B_{1}^{\widehat{H}_0}(0)\cap\mathcal{C}$, we let
 \begin{equation*}
    S(t):=x(t)+R(t)\omega .
 \end{equation*}
 Then we have $\widehat{H}_0(S(t)-x(t))=R(t)$, thus $S(t)\in\partial B_{R(t)}^{\widehat{H}_0}(x(t))\cap\mathcal{C}=\{x\in\mathcal{C}:u(x)=t\}$. It follows that $u(S(t))=t$ for any $t<t_0$ and $\omega\in\partial B_{1}^{\widehat{H}_0}(0)\cap\mathcal{C}$. By direct computations, we deduce that
  \begin{align}\label{eq:6.12}
    \nu(S(t))=-\frac{\nabla u(S(t))}{\left\lvert \nabla u(S(t))\right\rvert }=
    \frac{\nabla \widehat{H}_0(S(t)-x(t))}{\lvert \nabla \widehat{H}_0(S(t)-x(t))\rvert }
    =\frac{\nabla \widehat{H}_0(\omega)}{\lvert \nabla \widehat{H}_0(\omega)\rvert }.
  \end{align}
  Thus one has
  \begin{align}\label{eq:6.13}
    H(-\nu(S(t)))=\frac{H(\nabla u(S(t)))}{\left\lvert \nabla u(S(t))\right\rvert}
    =\frac{\widehat{H}(\nabla \widehat{H}_0(\omega))}{\lvert \nabla \widehat{H}_0(\omega)\rvert }
    =\frac{1}{\lvert \nabla \widehat{H}_0(\omega)\rvert}.
  \end{align}
  Noting that $H(\nabla u(S(t)))\equiv C>0$ for any $t<t_0$ and $\omega\in\partial B_{1}^{\widehat{H}_0}(0)\cap\mathcal{C}$, by \eqref{eq:6.13}, we have
  \begin{equation}\label{eq:6.14}
    \left\lvert \nabla u(S(t))\right\rvert
    =H(\nabla u(S(t)))\lvert \nabla \widehat{H}_0(\omega)\rvert
    =C\lvert \nabla \widehat{H}_0(\omega)\rvert
  \end{equation}
  for any $t<t_0$ and $\omega\in\partial B_{1}^{\widehat{H}_0}(0)\cap\mathcal{C}$. By \eqref{eq:6.12} and \eqref{eq:6.14}, we know
  that
  \begin{equation}\label{eq:6.15}
    \nabla u(S(t))=-C\nabla \widehat{H}_0(\omega)
  \end{equation}
  holds for $t<t_0$ and $\omega\in\partial B_{1}^{\widehat{H}_0}(0)\cap\mathcal{C}$. Moreover, \eqref{eq:6.15} yields
  \begin{equation}\label{eq:6.16}
    \left\langle \nabla u(S(t)),\omega\right\rangle
    =-C\langle \nabla \widehat{H}_0(\omega),\omega\rangle
    =-C\widehat{H}_0(\omega)
    =-C.
  \end{equation}

\medskip

Let
\begin{equation*}
    X(t):=x(t)+R(t)\omega^+\ \quad \text{and}\ \quad Y(t):=x(t)+R(t)\omega^-,
\end{equation*}
where $\omega ^{+}:=\frac{x_1}{\widehat{H}_0(x_1)}$ and $\omega ^{-}:=\frac{-x_1}{H_0(x_1)}$. Noting that $\widehat{H}_0(X(t)-x(t))=R(t)$ and $\widehat{H}_0(Y(t)-x(t))=R(t)$, thus we have
$X(t), Y(t)\in\partial B_{R(t)}^{\widehat{H}_0}(x(t))$. Moreover, one has $X(t), Y(t)\in\mathbb{R}^k\times\{0_{\mathbb{R}^{N-k}}\}$.
Since $H_0\in C^2(\mathbb{R}^N\setminus \{0\})$, letting $t=t_1$ in \eqref{eq:6.15} and then letting $\omega\rightarrow\omega^+,\omega^-$ respectively, we deduce that
\begin{equation}\label{eq:6.17}
    \lim\limits_{\omega\rightarrow\omega^+}\nabla u(S(t_1))=-C\nabla \widehat{H}_0(\omega^+)
\end{equation}
and
\begin{equation}\label{eq:6.18}
    \lim\limits_{\omega\rightarrow\omega^-}\nabla u(S(t_1))=-C\nabla \widehat{H}_0(\omega^-).
\end{equation}
Noting that $\lim\limits_{\omega\rightarrow\omega^+}S(t_1)=X(t_1)$ and $\lim\limits_{\omega\rightarrow\omega^-}S(t_1)=Y(t_1)$, we can define
\begin{equation}\label{eq:6.19}
    \nabla u(X(t_1)):=  \lim\limits_{\omega\rightarrow\omega^+}\nabla u(S(t_1))=-C\nabla \widehat{H}_0(\omega^+)
\end{equation}
and
\begin{equation}\label{eq:6.20}
    \nabla u(Y(t_1)):=  \lim\limits_{\omega\rightarrow\omega^-}\nabla u(S(t_1))=-C\nabla \widehat{H}_0(\omega^-).
\end{equation}
From \eqref{eq:6.19} and \eqref{eq:6.20}, we infer that
\begin{equation}\label{eq:6.21}
    \langle \nabla u(X(t_1)),\omega^+\rangle
    =-C\langle \nabla \widehat{H}_0(\omega^+),\omega^+\rangle
    =-C\widehat{H}_0(\omega^+)
    =-C
\end{equation}
and
\begin{equation}\label{eq:6.22}
    \langle \nabla u(Y(t_1)),\omega^-\rangle
    =-C\langle \nabla \widehat{H}_0(\omega^-),\omega^-\rangle
    =-C\widehat{H}_0(\omega^-)
    =-C.
\end{equation}
Through direct calculations, the definitions of $S(t)$, $\omega^+$ and $\omega^-$ yield that
\begin{equation}\label{eq:6.23}
    \begin{aligned}
        1=\frac{d}{dt}u(S(t))\bigg |_{t=t_1}
        &=\left\langle \nabla u(S(t_1)),S^{\prime}(t_1)\right\rangle \\
        &=\left\langle \nabla u(S(t_1)),x_1+R^{\prime}(t_1)\omega\right\rangle\\
        &= \widehat{H}_0(x_1)\left\langle \nabla u(S(t_1)),\omega^+\right\rangle+R^{\prime}(t_1)\left\langle \nabla u(S(t_1)),\omega\right\rangle
    \end{aligned}
\end{equation}
and
\begin{equation}\label{eq:6.24}
    \begin{aligned}
        1=\frac{d}{dt}u(S(t))\bigg |_{t=t_1}
        &=\left\langle \nabla u(S(t_1)),S^{\prime}(t_1)\right\rangle \\
        &=\left\langle \nabla u(S(t_1)),x_1+R^{\prime}(t_1)\omega\right\rangle\\
        &= H_0(x_1)\left\langle \nabla u(S(t_1)),-\omega^-\right\rangle+R^{\prime}(t_1)\left\langle \nabla u(S(t_1)),\omega\right\rangle.
    \end{aligned}
\end{equation}
Letting $\omega\rightarrow\omega^+$ in \eqref{eq:6.23} and $\omega\rightarrow\omega^-$ in \eqref{eq:6.24}
respectively,  by \eqref{eq:6.19}--\eqref{eq:6.22} we deduce that
\begin{equation}\label{eq:6.25}
    1=-C(\widehat{H}_0(x_1)+R^{\prime}(t_1))
\end{equation}
and
\begin{equation}\label{eq:6.26}
    1=-C\left(-H_0(x_1)+R^{\prime}(t_1)\right).
\end{equation}
Combining \eqref{eq:6.25} and \eqref{eq:6.26}, we conclude that
\begin{equation*}
    \widehat{H}_0(x_1)+H_0(x_1)=0,
\end{equation*}
which contradicts $x_1\neq0$. Thus we deduce that, $x'(t)\equiv0$ for $t<t_0$, and hence there exists $x_0\in\mathbb{R}^k\times\{0_{\mathbb{R}^{N-k}}\}$ such that
 $x(t)\equiv x_0$ for $t<t_0$. Finally, by \eqref{eq:6.11}, we arrive at
 \begin{equation*}
    e^{u(x)}=\frac{c_N\lambda^N}{\left[1+\lambda^{\frac{N}{N-1}}\widehat{H}_0^{\frac{N}{N-1}}(x-x_0)\right]^N},
 \end{equation*}
and hence \eqref{eq:1.2} follows immediately. This concludes our proof of Theorem \ref{Th:1.1}.  \qquad\qquad\qquad\quad$\qed$

\appendix
\section{The radial Poincar\'{e} type inequality and its applications}\label{appendix}
In this section, we will prove the radial Poincar\'{e} type inequality. Then, as applications, we will prove the Brezis-Merle type exponential inequality for mixed boundary value problems of Finsler $N$-Laplacian in Lemmas \ref{le:1+} and \ref{le:1+a}.

\medskip

First, we are to prove the radial Poincar\'{e} type inequality. Let $\Omega\subset\mathbb{R}^{N}$ be an open set. For any point $P\in\mathbb{R}^{N}$, we define the largest cross-section
$$\Sigma^{P}_{\Omega}:=\left(\bigcup\limits_{0<\rho<+\infty}\rho\left(\Omega-\{P\}\right)\right)\cap\mathbb{S}^{N-1}=\left\{\frac{x-P}{|x-P|}\,\Bigg|\,x\in\Omega,\,x\neq P\right\}$$
and the open cone $\mathcal{C}_{\Sigma^{P}_{\Omega},P}:=\{P\}+\bigcup\limits_{0<r<+\infty}r\Sigma^{P}_{\Omega}$ with vertex $P$ and cross-section $\Sigma^{P}_{\Omega}$, where $\mathbb{S}^{N-1}:=\partial B_{1}(0)$. One can easily see that $\Omega\subset\mathcal{C}_{\Sigma^{P}_{\Omega},P}$. Assume that there is a radial center $P\in\mathbb{R}^{N}$ (usually not unique) such that $\Omega$ satisfies the following property $\emph{(P)}$:

\medskip

\noindent $\emph{(P)}$ \, for any unit vector $\omega\in\Sigma^{P}_{\Omega}$, there exist an integer $K_{\omega}\geq1$ and a sequence of radii $0\leq \underline{r}^{1}_{\omega}\leq \overline{r}^{1}_{\omega}<\underline{r}^{2}_{\omega}\leq \overline{r}^{2}_{\omega}<\cdots<\underline{r}^{K_{\omega}}_{\omega}\leq \overline{r}^{K_{\omega}}_{\omega}$ such that
$$
\left(\bigcup\limits_{0\leq r<+\infty}\{r\omega+P\}\right)\cap\partial\Omega=\bigcup\limits_{1\leq k\leq K_{\omega}}\{\overline{r}^{k}_{\omega}\omega,\underline{r}^{k}_{\omega}\omega\}.
$$

\begin{defn}
Define $\Gamma_{P}:=\partial\Omega\cap\mathcal{C}_{\Sigma^{P}_{\Omega},P}$, then $\partial\Omega\setminus\Gamma_{P}=\partial\mathcal{C}_{\Sigma^{P}_{\Omega},P}\cap\partial\Omega$. We denote the (front) radial contact set in X-rays from the radial center $P$ by $$\Gamma^{-}_{P}:=\{\underline{r}^{k}_{\omega}\omega\mid\,1\leq k\leq K_{\omega},\,\omega\in\Sigma^{P}_{\Omega}\}\subset\Gamma_{P}\subset\partial\Omega,$$
and the (back) radial contact set in X-rays from the radial center $P$ by
$$\Gamma^{+}_{P}:=\Gamma_{P}\setminus\Gamma^{-}_{P}=\{\overline{r}^{k}_{\omega}\omega \mid\,1\leq k\leq K_{\omega}, \, \omega\in\Sigma^{P}_{\Omega}\}.$$
The notation
\begin{equation}\label{rw}
  d_{r,P}(\Omega):=\sup\limits_{\omega\in\Sigma^{P}_{\Omega}}\max\limits_{1\leq k\leq K_{\omega}}\left(\overline{r}^{k}_{\omega}-\underline{r}^{k}_{\omega}\right)
\end{equation}
denotes the radial width of $\Omega$ with respect to the radial center $P$.
\end{defn}

\begin{figure}[H]\label{F}
		\centering	
					\begin{tikzpicture}
					\node[left] at (0:0){$P$};
                    \filldraw[fill=gray!10, draw=gray!10] (16:4)--(16:8) arc (16:-16:8)--(-16:4) [shift=(0:7.7), rotate=180] (16:4) arc (16:-16:4);
                    \draw[bline, line width=1.5] (16:8) arc (16:-16:8);
                    \draw[shift=(0:7.7), rotate=180, rline, line width=1.5] (16:4) arc (16:-16:4);
					\draw (0:0) -- (16:9) (0:0) -- (-16:9);
					\node[below]at(0:10.5) {$\mathcal C_{\Sigma^{P}_{\Omega},P}$};
                    \draw[dashed] (0:0)--(0:3.7) (0:8)--(0:12);
					\draw[dashed] (16:9)--(16:12) (-16:9)--(-16:12);
                    \draw[->] (0:5.2)--(0:3.72);
                    \draw[->] (0:6.8)--(0:7.98); \node[] at (0:6) {$d_{r,P}(\Omega)$};
                    \draw[->] (8:10)--(2:8.02); \node[above] at (8:11) {The (back) radial contact set $\Gamma^{+}_{P}$};
                    \draw[->] (50:3)--(2:3.67); \node[above] at (50:3) {The (front) radial contact set $\Gamma^{-}_{P}$};
                    \draw[->] (-23:5.6)--(-8:7.98);
                    \draw[->] (-23:5.6)--(-8:3.82); \node[below] at (-20:6.5) {$\Gamma_{P}=\Gamma^{+}_{P}\cup\Gamma^{-}_{P}$};
					\draw[line width=1.5] (16:2) arc (16:-16:2);
                    \draw[->] (30:1)--(2:1.97) ;\node[above] at (30:1) {$\Sigma^{P}_{\Omega}+\{P\}$};
                    \node[above] at (3:6){$\Omega$};
				\end{tikzpicture}
		\centering
		\caption{The (front) radial contact set $\Gamma^{-}_{P}$ and (back) radial contact set $\Gamma^{+}_{P}$ of $\Omega$ in X-rays from the radial center $P$}
	\end{figure}
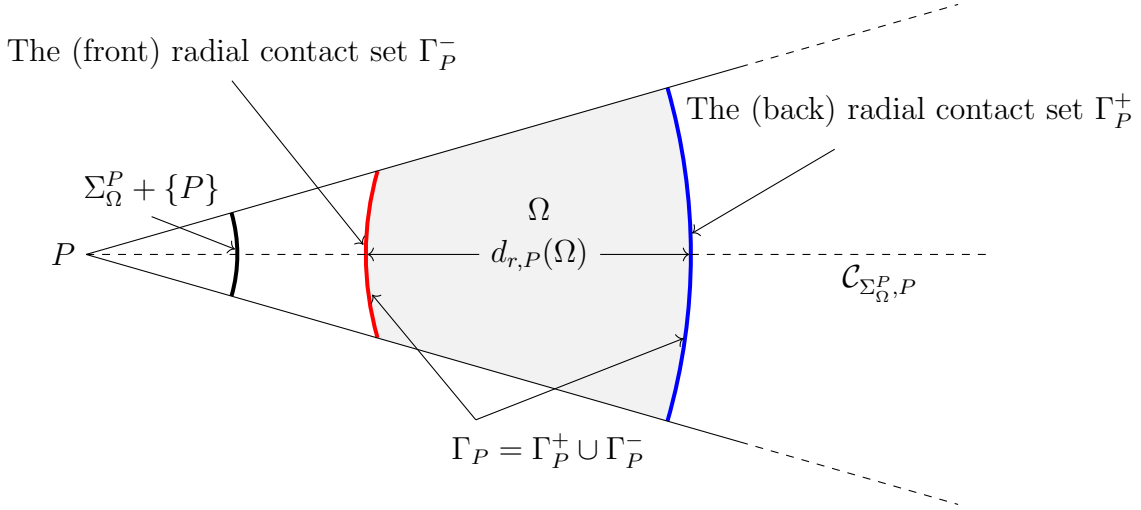

\begin{rem}\label{eg}
For instance, if $\Omega=\left(B_{R}(0)\setminus B_{r}(0)\right)\cap\mathcal{C}$ is a fan-shaped domain for some open cone $\mathcal{C}\subset\mathbb{R}^{N}$ with vertex at $0$ and $0\leq r<R<+\infty$, then $\Omega$ satisfies the property $(P)$ for radial center $P=0$, and one has $\Gamma^{+}_{0}=\partial B_{R}(0)\cap\mathcal{C}$, $\Gamma^{-}_{0}=\partial B_{r}(0)\cap\mathcal{C}$ and $d_{r,0}(\Omega)=R-r$. If $\Omega=\left(\left(B_{8}(0)\setminus B_{4}(0)\right)\cup\left(B_{2}(0)\setminus B_{1}(0)\right)\right)\cap\mathcal{C}$ for some open cone $\mathcal{C}\subset\mathbb{R}^{N}$ with vertex at $0$, then $\Omega$ satisfies the property $(P)$ for radial center $P=0$, and one has $\Gamma^{+}_{0}=\left(\partial B_{8}(0)\cup\partial B_{2}(0)\right)\cap\mathcal{C}$, $\Gamma^{-}_{0}=\left(\partial B_{4}(0)\cup\partial B_{1}(0)\right)\cap\mathcal{C}$ and $d_{r,0}(\Omega)=4$.
\end{rem}

We have the following radial Poincar\'{e} type inequality, which is of its own importance and interests.
\begin{lem}[Radial Poincar\'{e} type inequality]\label{A1}
Let $N\geq2$ and $\Omega\subset\mathbb{R}^{N}$ be an open set. Assume that there exists a radial center $P\in\mathbb{R}^{N}$ (usually not unique) such that, $\Omega$ satisfies property $\emph{(P)}$. Then, for any $1\leq p<+\infty$, we have
\begin{equation}\label{rp}
  \|f\|_{L^{p}(\Omega)}\leq d_{r,P}(\Omega)\|\nabla f\|_{L^{p}(\Omega)}
\end{equation}
for any $f\in W^{1,p}(\Omega)$ with $f=0$ on the (back) radial contact set $\Gamma^{+}_{P}$.
\end{lem}
\begin{proof}
If $d_{r,P}(\Omega)=+\infty$, then inequality \eqref{rp} is trivial. We assume $d_{r,P}(\Omega)<+\infty$ from now on. Without loss of generality, we may assume that the radial center $P=0$. Let $\det\, J=\frac{\partial(x_{1},x_{2},\cdots,x_{N})}{\partial(r,\omega_{1},\cdots,\omega_{N-1})}$ denote the Jacobian determinant for the transformation of spherical coordinates, then $\det\, J(r,\omega)=r^{n-1}\det\, J(1,\omega)$. Note that, for any $x\in\Omega$ ($x\neq0$), one has $\omega:=\frac{x}{|x|}\in\Sigma^{0}_{\Omega}$. Thus there exists a unique $1\leq k\leq K_{\omega}$ depending on $x$, such that $r=|x|\in(\underline{r}^{k}_{\omega},\overline{r}^{k}_{\omega})$. Note that $f=0$ on $\Gamma^{+}_{0}$, by Fubini's theorem and H\"{o}lder's inequality, we have
\begin{eqnarray*}
  &&\int_{\Omega}|f(x)|^{p}\mathrm{d}x=\int_{\Sigma^{P}_{\Omega}}\mathrm{d}\omega\sum_{k=1}^{K_{\omega}}\int_{\underline{r}^{k}_{\omega}}^{\overline{r}^{k}_{\omega}}|\det\, J(r,\omega)|\mathrm{d}r\bigg|\int_{r}^{\overline{r}^{k}_{\omega}}\frac{\partial f}{\partial r}(s,\omega)\mathrm{d}s\bigg|^{p} \\
  &&\qquad\qquad\quad \leq\int_{\Sigma^{P}_{\Omega}}\mathrm{d}\omega\sum_{k=1}^{K_{\omega}}\int_{\underline{r}^{k}_{\omega}}^{\overline{r}^{k}_{\omega}}\left(\overline{r}^{k}_{\omega}-r\right)^{p-1}|\det\, J(r,\omega)|\mathrm{d}r\int_{r}^{\overline{r}^{k}_{\omega}}\Big|\frac{\partial f}{\partial r}(s,\omega)\Big|^{p}\mathrm{d}s \\
   && \qquad\qquad\quad \leq [d_{r,0}(\Omega)]^{p-1}\int_{\Sigma^{P}_{\Omega}}\mathrm{d}\omega\sum_{k=1}^{K_{\omega}}\int_{\underline{r}^{k}_{\omega}}^{\overline{r}^{k}_{\omega}}|\nabla f(s,\omega)|^{p}\mathrm{d}s\int^{s}_{\underline{r}^{k}_{\omega}}r^{n-1}|\det\, J(1,\omega)|\mathrm{d}r \\
   && \qquad\qquad\quad \leq [d_{r,0}(\Omega)]^{p-1}\int_{\Sigma^{P}_{\Omega}}\mathrm{d}\omega\sum_{k=1}^{K_{\omega}}\int_{\underline{r}^{k}_{\omega}}^{\overline{r}^{k}_{\omega}}
  \left(s-\underline{r}^{k}_{\omega}\right)|\nabla f(s,\omega)|^{p}\cdot s^{n-1}|\det\, J(1,\omega)|\mathrm{d}s \\
  && \qquad\qquad\quad \leq [d_{r,0}(\Omega)]^{p-1}\int_{\Sigma^{P}_{\Omega}}\mathrm{d}\omega\sum_{k=1}^{K_{\omega}}\left(\overline{r}^{k}_{\omega}-\underline{r}^{k}_{\omega}\right)\int_{\underline{r}^{k}_{\omega}}^{\overline{r}^{k}_{\omega}}
  |\nabla f(s,\omega)|^{p}\cdot|\det\, J(s,\omega)|\mathrm{d}s \\
  &&\qquad\qquad\quad \leq [d_{r,0}(\Omega)]^{p}\int_{\Omega}|\nabla f|^{p}\mathrm{d}x.
\end{eqnarray*}
This finishes our proof of the radial Poincar\'{e} type inequality in Lemma \ref{A1}.
\end{proof}
\begin{rem}
Since $\Gamma^{+}_{P}\subset\Gamma_{P}\subset\partial\Omega$, so $f=0$ on $\partial\Omega$ or $\Gamma_{P}$ implies that $f=0$ on $\Gamma^{+}_{P}$, thus the usually well-known Poincar\'{e} inequality for $f\in W^{1,p}_{0}(\Omega)$ is a direct corollary of the radial Poincar\'{e} type inequality in Lemma \ref{A1}.
\end{rem}

As application of the radial Poincar\'{e} type inequality, we can prove the Brezis-Merle type exponential inequality for mixed boundary value problems of Finsler $N$-Laplacian, which has its own interests and importance and is a variant and generalization of Lemma \ref{le:1}. One should note that, in the following Lemma \ref{le:1+a}, we only require $u=h$ on the boundary subset $\Gamma$ containing the (back) radial contact set $\Gamma^{+}_{P}$ (i.e., $\Gamma^{+}_{P}\subseteq\Gamma\subset\partial\Omega$), while Lemma \ref{le:1} requires $u=h$ on the whole boundary $\partial\Omega$.
\begin{lem}[Brezis-Merle type exponential inequality for mixed boundary value problems: a variant of Lemma \ref{le:1}]\label{le:1+a}
    Let $N\geq2$ and $\Omega\subset\mathbb{R}^{N}$ be a bounded domain. Assume that there exists a radial center $P\in\mathbb{R}^{N}$ (usually not unique) such that $\Omega$ satisfies property $\emph{(P)}$. Let the gauge $H$ be the same as in Lemma \ref{8,le:2.3}. Assume that $u,h\in W^{1,N}(\Omega)$ are weak solutions of
    \begin{equation*}
     \left\{
         \begin{aligned}
             &-\Delta ^{H}_{N}u=f, \quad \,\, -\Delta ^{H}_{N}h=0 \ &\rm{in}\ \Omega, \\
             &u=h \ &\rm{on}\ \Gamma,\\
             &\langle a(\nabla u),\nu\rangle=\langle a(\nabla h),\nu\rangle=0\ &\rm{on}\ \partial\Omega\setminus\Gamma,
             \end{aligned}
             \right.
    \end{equation*}
    where $f\in L^1(\Omega)$, the boundary subset $\Gamma$ satisfies $\Gamma^{+}_{P}\subseteq\Gamma\subset\partial\Omega$ and $\Gamma^{+}_{P}$ is the (back) radial contact set. Then there exists a positive constant $\mu_\Omega$ (given by \eqref{ce}), depending only on $N$, $H$, $\Omega$ and $P$, such that for every $0<\lambda<\mu_\Omega \|f\|_{L^1(\Omega)}^{-\frac{1}{N-1}}$, it holds
    \begin{equation*}
        \int_{\Omega}e^{\lambda|u-h|}\mathrm{d}x\leq \frac{\mathcal{L}^N(\Omega)}{1-\lambda\mu_{\Omega}^{-1}\|f\|_{L^1(\Omega)}^{\frac{1}{N-1}}}.
    \end{equation*}
 \end{lem}
\begin{proof}
Fix $k\geq0$, $a>0$, define the truncation operator $T_k$ as
    \begin{equation*}
        T_k(u)=
        \left\{
            \begin{aligned}
                &u,\ \ &\text{if}\ |u|\leq k,\\
                &k\frac{u}{|u|},\ \ &\text{if}\ |u|\geq k.
                \end{aligned}
                \right.
    \end{equation*}
Since $u=h$ on $\Gamma$, then $T_{k+a}(u-h)=T_k(u-h)=0$ on $\Gamma$. Note that $|T_k(w)|\leq k$ holds for any $w\in W^{1,N}(\Omega)$, then one has
$$T_{k+a}(u-h)-T_{k}(u-h)\in W^{1,N}(\Omega)\cap L^{\infty}(\Omega).$$
Choosing $T_{k+a}(u-h)-T_{k}(u-h)$ as a test function in
\begin{equation*}
    \left\{
        \begin{aligned}
            &-\Delta ^{H}_{N}u=f, \quad \,\, -\Delta ^{H}_{N}h=0 \ &\rm{in}\,\, \Omega, \\
            &u=h \ &\rm{on}\,\, \Gamma,\\
            &\langle a(\nabla u),\nu\rangle=\langle a(\nabla h),\nu\rangle=0\quad\,\, &\rm{on}\,\, \partial\Omega\setminus\Gamma,
            \end{aligned}
            \right.
   \end{equation*}
we have
\begin{equation*}
    \begin{aligned}
        &\int_{\Omega}\langle a(\nabla u), \nabla[T_{k+a}(u-h)-T_{k}(u-h)]\rangle\mathrm{d}x+\int_{\Gamma}\langle a(\nabla u), \nu\rangle[T_{k+a}(u-h)-T_{k}(u-h)]\mathrm{d}\mathcal{H}^{N-1}\\
        &+\int_{\partial\Omega\setminus\Gamma}\langle a(\nabla u), \nu\rangle[T_{k+a}(u-h)-T_{k}(u-h)]\mathrm{d}\mathcal{H}^{N-1}=\int_{\Omega}f[T_{k+a}(u-h)-T_{k}(u-h)]\mathrm{d}x,\\
    \end{aligned}
\end{equation*}
and
\begin{equation*}
    \begin{aligned}
        &\int_{\Omega}\langle a(\nabla h), \nabla[T_{k+a}(u-h)-T_{k}(u-h)]\rangle\mathrm{d}x+\int_{\Gamma}\langle a(\nabla h), \nu\rangle[T_{k+a}(u-h)-T_{k}(u-h)]\mathrm{d}\mathcal{H}^{N-1}\\
        &+\int_{\partial\Omega\setminus\Gamma}\langle a(\nabla h), \nu\rangle[T_{k+a}(u-h)-T_{k}(u-h)]\mathrm{d}\mathcal{H}^{N-1}=0.\\
    \end{aligned}
\end{equation*}
By the boundary conditions $u=h$ on $\Gamma$ and $\langle a(\nabla u),\nu\rangle=\langle a(\nabla h),\nu\rangle=0$ on $\partial\Omega\setminus\Gamma$, we have
\begin{equation*}
\int_{\Omega}\langle a(\nabla u), \nabla[T_{k+a}(u-h)-T_{k}(u-h)]\rangle\mathrm{d}x=\int_{\Omega}f[T_{k+a}(u-h)-T_{k}(u-h)]\mathrm{d}x
\end{equation*}
and
\begin{equation*}
    \int_{\Omega}\langle a(\nabla h), \nabla[T_{k+a}(u-h)-T_{k}(u-h)]\rangle\mathrm{d}x=0.
    \end{equation*}
Thus we obtain that
    \begin{equation}\label{eq:2.8+}
        \int_{\Omega}\langle a(\nabla u)-a(\nabla h), \nabla[T_{k+a}(u-h)-T_{k}(u-h)]\rangle\mathrm{d}x=\int_{\Omega}f[T_{k+a}(u-h)-T_{k}(u-h)]\mathrm{d}x.
    \end{equation}
Moreover, for any $w\in W^{1,N}(\Omega)$, by definition of the truncation operator $T_{k}$, we have
\begin{equation*}
    T_{k+a}(w)-T_k(w)=
    \left\{
    \begin{aligned}
            & 0, \ \ & \, |w|\leq k, \\
            & w-k\frac{w}{|w|}, \ \ & \, k<|w|\leq k+a, \\
            & a\frac{w}{|w|}, \ \ & \, |w|>k+a,
    \end{aligned}
    \right.
\end{equation*}
and hence, it follows that
\begin{equation*}
    \nabla [T_{k+a}(w)-T_k(w)]=
    \left\{
    \begin{aligned}
            &\nabla w, \ \ &k<|w|\leq k+a,\\
            &0,\ \ &|w|\leq k\ \text{or}\ |w|>k+a,
    \end{aligned}
    \right.
\end{equation*}
and
$$|T_{k+a}(w)-T_k(w)|\leq a.$$
Letting $w=u-h$, $\xi_1=\nabla u$ and $\xi_2=\nabla h$ in \eqref{aeq:2.4}, we deduce from \eqref{eq:2.8+} that
\begin{equation}\label{eq:2.9+}
    \frac{1}{a}\int_{\{k<|u-h|\leq k+a\}}|\nabla(u-h)|^N\mathrm{d}x\leq\frac{\lVert f\rVert_{L^1(\Omega)}}{c_1}.
\end{equation}
Let $\Phi(k):=\mathcal{L}^N(\{x\in\Omega\mid\,|w(x)|>k\})$ be the distribution function of $w$, where $w=u-h$. Since $\mathcal{C}$ is a cone with vertex at the origin $0$ and $w=u-h=0$ on $\Gamma$, by applying the radial Poincar\'{e} type inequality in Lemma \ref{A1} with $p=1$, we deduce that
\begin{equation*}
    \begin{aligned}
        \Phi(k+a)^{\frac{N-1}{N}}
        &=\frac{1}{a}\left(\int_{\{|w|>k+a\}}\left\lvert (k+a)\frac{w}{|w|}-a\frac{w}{|w|}\right\rvert^{\frac{N}{N-1}}\mathrm{d}x \right)^{\frac{N-1}{N}}\\
        &=\frac{1}{a}\left(\int_{\{|w|>k+a\}}|T_{k+a}(w)-T_k(w)|^{\frac{N}{N-1}}\mathrm{d}x\right)^{\frac{N-1}{N}} \\
        &\leq \frac{1}{a}\left(\int_{\Omega}|T_{k+a}(w)-T_k(w)|^{\frac{N}{N-1}}\mathrm{d}x\right)^{\frac{N-1}{N}}\\
        &\leq\frac{S(\Omega)}{a} \left(\int_{\Omega}|T_{k+a}(w)-T_k(w)|\mathrm{d}x+\int_{\Omega}|\nabla [T_{k+a}(w)-T_k(w)]|\mathrm{d}x\right) \\
        &\leq\frac{S(\Omega)}{a}\left(d_{r,P}(\Omega)+1\right)\int_{\{k<|w|\leq k+a\}}|\nabla w|\mathrm{d}x,
    \end{aligned}
\end{equation*}
where $S(\Omega)$ is the Sobolev constant of the embedding $W^{1,1}(\Omega)\hookrightarrow L^{\frac{N}{N-1}}(\Omega)$, and $d_{r,P}(\Omega)$ defined by \eqref{rw} denotes the radial width of $\Omega$ with respect to the radial center $P$. By \eqref{eq:2.9+} and H\"{o}lder's inequality, we have
$$\Phi(k+a)\leq\frac{\Phi(k)-\Phi(k+a)}{a\Lambda},$$
where $\Lambda:=\left(\frac{c_1^{\frac{1}{N}}}{S(\Omega)(d_{r,P}(\Omega)+1)}\right)^{\frac{N}{N-1}}\left\lVert f\right\rVert^{-\frac{1}{N-1}}_{L^1(\Omega)}$. Letting $a\rightarrow0^+$, we have
\begin{equation}\label{eq:2.10+}
    \Phi(k)\leq-\frac{1}{\Lambda}\Phi^{\prime}(k)\quad\,\,\,\text{for}\,\,\,\text{a.e.}\,\,k>0.
\end{equation}
Since $\Phi$ is monotonically decreasing, by integrating \eqref{eq:2.10+} from $0$ to $k$, we get
$$\log\left[\frac{\Phi(k)}{\Phi(0)}\right]\leq\int_{0}^{k}\frac{\Phi^{\prime}(s)}{\Phi(s)}\mathrm{d}s\leq-\Lambda k,$$
which implies that
$$\Phi(k)\leq \mathcal{L}^N(\Omega)e^{-\Lambda k}\qquad \ \text{for all}\ \ k>0.$$
Then, we have
\begin{equation*}
    \begin{aligned}
        \int_{\Omega}e^{\lambda|w|}\mathrm{d}x&=\mathcal{L}^N(\Omega)+\lambda\int_{\Omega}\mathrm{d}x\int_{0}^{|w(x)|}e^{\lambda k}\mathrm{d}k=\mathcal{L}^N(\Omega)+\lambda\int_{0}^{\infty}e^{\lambda k}\Phi(k)\mathrm{d}k\\
        &\leq \mathcal{L}^N(\Omega)+\lambda\mathcal{L}^N(\Omega)\int_{0}^{\infty}e^{-(\Lambda-\lambda) k}\mathrm{d}k=\frac{\mathcal{L}^N(\Omega)}{1-\lambda\Lambda^{-1}}.
    \end{aligned}
\end{equation*}
Recall that $w=u-h$ and $\Lambda=\left(\frac{c_1^{\frac{1}{N}}}{S(\Omega)\left(d_{r,P}(\Omega)+1\right)}\right)^{\frac{N}{N-1}}\left\lVert f\right\rVert^{-\frac{1}{N-1}}_{L^1(\Omega)}$. Letting
\begin{equation}\label{ce}
  \mu_E:=\left(\frac{c_1^{\frac{1}{N}}}{S(\Omega)\left(d_{r,P}(\Omega)+1\right)}\right)^{\frac{N}{N-1}},
\end{equation}
we finished our proof of Lemma \ref{le:1+a}.
\end{proof}

\begin{rem}\label{rem-a}
By letting $\Omega=E\cap\mathcal{C}$ (where $\mathcal{C}$ is an arbitrary open cone with vertex at the origin $0$ and $E$ is a bounded domain), $P=0$ and $\Gamma=\partial E\cap\mathcal{C}=\Gamma_{0}$ in Lemma \ref{le:1+a}, we can deduce the Brezis-Merle type exponential inequality for mixed boundary value problems of Finsler $N$-Laplacian in general cones $\mathcal{C}$ from Lemma \ref{le:1+a} directly. By letting $\Gamma=\Gamma^{+}_{P}$ in Lemma \ref{le:1+a}, we get the Brezis-Merle type exponential inequality for mixed boundary value problem which only requires $u=h$ on the (back) radial contact set $\Gamma^{+}_{P}$.
\end{rem}

Besides the important application in the proof of Lemma \ref{le:1+a}, the radial Poincar\'{e} type inequality in Lemma \ref{A1} also yields an interesting observation that, the Poincar\'{e} inequality holds in $B_{1}(0)$ with constant $1$ if $f=0$ on the whole unit sphere $\mathbb{S}^{N-1}$, and holds in $B_{1}(0)$ with constant $2$ if $f=0$ on ``almost half" of the unit sphere $\mathbb{S}^{N-1}$.
\begin{cor}\label{cor}
Let $N\geq 2$ and $1\leq p<+\infty$. Then we have $\|f\|_{L^{p}(B_{1}(0))}\leq \|\nabla f\|_{L^{p}(B_{1}(0))}$ for any $f\in W^{1,p}_{0}(B_{1}(0))$. For arbitrarily small $0<\epsilon<1$, we have $\|f\|_{L^{p}(B_{1}(0))}\leq 2\|\nabla f\|_{L^{p}(B_{1}(0))}$ for any $f\in W^{1,p}(B_{1}(0))$ with $f=0$ on $\partial B_{1}(0)\cap \{x\mid x_{N}>-\epsilon\}$.
\end{cor}
\begin{proof}
By choosing the radial center $P=0$, it follows immediately from the radial Poincar\'{e} type inequality in Lemma \ref{A1} that $\|f\|_{L^{p}(B_{1}(0))}\leq \|\nabla f\|_{L^{p}(B_{1}(0))}$ for any $f\in W^{1,p}_{0}(B_{1}(0))$. For arbitrarily small $0<\epsilon<1$, we choose the radial center $P=(0,\cdots,0,-M)$ with $M>0$ sufficiently large such that $\Gamma^{+}_{P}\subset\partial B_{1}(0)\cap \{x\mid x_{N}>-\epsilon\}$, then $d_{r,P}(B_{1}(0))=2$ and $f=0$ on $\Gamma^{+}_{P}$ provided that $f=0$ on $\partial B_{1}(0)\cap \{x\mid x_{N}>-\epsilon\}$. Therefore, the radial Poincar\'{e} type inequality in Lemma \ref{A1} implies that $\|f\|_{L^{p}(B_{1}(0))}\leq 2\|\nabla f\|_{L^{p}(B_{1}(0))}$ for any $f\in W^{1,p}(B_{1}(0))$ with $f=0$ on $\partial B_{1}(0)\cap \{x\mid x_{N}>-\epsilon\}$. This finishes the proof of Corollary \ref{cor}.
\end{proof}

We believe that the radial Poincar\'{e} type inequality in Lemma \ref{A1} can be applied to many problems in functional analysis, PDEs and geometric analysis and so on.

\end{document}